\documentclass[12pt]{article}
\usepackage{amssymb}
\usepackage{amssymb,amsfonts}
\usepackage{amsmath}
\usepackage{cases}

\usepackage{indentfirst}
\usepackage[mathscr]{eucal}
\usepackage[OT1]{fontenc}

\usepackage[round,sort&compress]{natbib}
\usepackage{lscape}
\usepackage{multirow}
\usepackage{graphicx}
\usepackage{graphics}
\usepackage{epsfig}
\usepackage{color,xcolor}
\usepackage{float}

\usepackage[colorlinks,linkcolor=blue,anchorcolor=blue,citecolor=blue]{hyperref}

\usepackage{url}

\usepackage[bf, small]{titlesec}
\newtheorem{theorem}{Theorem}

\newtheorem{lemma}{Lemma}[section]

\newtheorem{corollary}{Corollary}[section]

\newcommand{\CG}[1]{\textcolor{magenta}{#1}}

\newcommand{\bm}[1]{\mbox{\boldmath{$#1$}}}

\newcommand{\trans}{^{\mbox{\tiny{T}}}}

\usepackage{geometry}
\date{}

\begin{document}
\bibliographystyle{plainnat}
%\bibliography{compare}
%\linenumbers
\begin{center}
{\Large\bf A test for the $k$ sample Behrens-Fisher problem in}\\
\bigskip
{\Large\bf high dimensional data}\\
\bigskip
\bigskip
Mingxiang Cao$^{a}$, Junyong Park$^{b,*}$, Daojiang He$^{a}$\\
\bigskip
\bigskip
    {\it\small  $^{a}$School of Mathematics $\&$ Computer Science, Anhui Normal University, Wuhu 241000, China}\\
    {\it\small $^{b}$Department of Mathematics and Statistics, University of
Maryland Baltimore County, 1000 Hilltop Circle, Baltimore, MD 21250,
USA.}\\

\footnotetext[1]{*Corresponding author.}

\footnotetext[2]{~~E-mail address:~\url{junpark@umbc.edu} (J.Y.
Park).}
\end{center}

\begin{abstract}
\noindent In this paper, the $k$ sample Behrens-Fisher problem is
investigated in high dimensional setting. We propose a new test statistic
and  demonstrate that the proposed test is expected to have more
powers than some existing test especially when sample sizes are unbalanced.
We provide theoretical investigation as well as numerical studies on
both sizes and powers of the proposed tests and existing test.
Both theoretical comparison of  the asymptotic power functions
and numerical studies show that the proposed test
tends to have more powers than existing test in many
cases of unbalanced sample sizes.
\end{abstract}


\thispagestyle{empty}

\section{Introduction}
In many contemporary applications, high and ultrahigh dimensional
data are increasingly available, such as molecular biology,
genomics, fMRI, finance and transcriptomics. A common feature for
high and ultrahigh dimensional data is that the data dimension is
larger or much larger than the sample size, the so called ``large
$p$, small $n$" phenomenon where $p$ is the data dimension and $n$
is the sample size. In high dimensional settings, classical methods
may be invalid, or not applicable at all. Hence, there has been
growing interest in developing testing procedures which are better
suited to deal with statistical problems in high dimensional
setting. Testing hypotheses in high dimension is one of important
issues in high dimensional data which has attracted a great deal of
attention in recent decades. In two sample testing in high
dimension, there have been numerous studies such as
\cite{Bai:Saranadasa:1996},   \cite{Srivastava:2008:2009:2013},
\cite{Chen:Qin:2010}, \cite{Aoshima:Yata:2011},
%\cite{Park:Ayyala:2012},
\cite{Park:Ayyala:2013}, \cite{Feng:2015},
\cite{Zhou:Kong:2015}, \cite{{Ma:Wang:2015}}, \cite{Ghosh:2016} and
\cite{Zhao:Xu:2016}. For multivariate analysis of variance (MANOVA),
see \cite{Fujikoshi:2004}, \cite{Schott:2007},
\cite{Srivastava:2007}, \cite{Cai:2014} and \cite{Cao:2015}. More
specifically, when there are $k$ groups and
$\bm{X}_{l1},\cdots,\bm{X}_{ln_{l}}$ represent $p \times 1$ random
samples from the $l$th group with unknown mean vector $\bm{\mu}_{l}$
and positive definite covariance matrix $\bm{\Sigma}_{l}$ for $l=1,
\ldots, k$, it is of interest to test
\begin{equation}\label{eq1}
H_{0}:\ \bm{\mu}_{1}=\cdots=\bm{\mu}_{k}\ \ \mbox{versus}\ \ H_{1}:\
H_{0}\ \mbox{is not true.}
\end{equation}
In particular, when all covariance matrices are homogeneous such as
$\bm{\Sigma}_{1}=\cdots=\bm{\Sigma}_{k}$, testing \eqref{eq1}
is known as MANOVA.
On the other hand,
\cite{Hu:Bai:2015} and \cite{Cao:2014} recently proposed the same
test statistic to test \eqref{eq1} when covariance matrices are not
necessarily homogeneous. This is also known as the $k$ sample
Behrens-Fisher (BF) problem which does not require
$\bm{\Sigma}_{1}=\cdots=\bm{\Sigma}_{k}$.
The homogeneity of covariance matrices is a strong condition
in practice. In fact, it is not straightforward to verify the
homogeneity of covariance matrices especially in high dimensional
data. Therefore, unless there is any strong evidence supporting the homogeneity of covariance matrices,
it is natural to allow different covariance matrices in practice.

The main goal of this paper is to propose a new test statistic in
the $k$ sample Behrens-Fisher problem. It will be shown that the
proposed test behaves differently from existing test such as
\cite{Hu:Bai:2015} when sample sizes are unbalanced. We will discuss
such differences between  the proposed test and the test in
\cite{Hu:Bai:2015} through both theoretical and numerical
comparisons under a variety of situations. We observe that the
proposed test has some advantage in powers compared to
\cite{Hu:Bai:2015} in many cases situations through theoretical and
numerical comparisons.

The remainder of the paper is organized as follows.
Section~\ref{sec:2} first presents conditions of statistical model.
Some notations used throughout the paper are defined and assumptions
are also announced for the theoretical study. In
Section~\ref{sec:3}, we give the new test statistic and investigate
its asymptotic behavior under $H_{0}$ and $H_{1}$.  Theoretical
comparisons and numerical studies on the proposed test and the Hu's
test are carried out in Section~\ref{sec:4}.  Concluding remarks is presented in \ref{sec:5}.

\section{Preliminaries}\label{sec:2}
In this section, we give notations and the statistical model for the
$k$ sample BF problem. Some assumptions are also illustrated.

\subsection{Notations}
The following notations will be used in subsequent
exposition. All vectors are column and $\bm{M}\trans$ denotes the
transpose of $\bm{M}$. All vectors and matrices are bold-faced. For two
sequences of real numbers $\{a_{n}\}$ and $\{b_{n}\}$, we write
$a_{n}=O(b_{n})$ if there exists a constant $c$ such that
$|a_{n}|\leq c|b_{n}|$ holds for all sufficiently large $n$, and
write $a_{n}=o(b_{n})$ if $\lim\limits_{n\to\infty}a_{n}/b_{n}=0$.
For a random sequence $Z_{n}$ and a random variable $Z$,
$Z_{n}\overset{pr}\longrightarrow Z$ and
$Z_{n}\overset{d}\longrightarrow Z$ denote $Z_{n}$ converges to $Z$
in probability and in distribution, respectively, as $n\to\infty$.

Let $\overline{\bm{X}}_{l}$ and $\bm{S}_{ln_{l}}$
 be the sample mean vector and sample
covariance matrix from the $l$th group  for $l=1,\ldots,k$.
Let $\overline{\bm{X}}$ be the pooled sample mean vector which is $\frac{1}{\sum_{l=1}^k n_l}\sum_{l=1}^k n_l \bm{X}_l$.
If we define  $n_{l1}:=[n_{l}/2]+1$ and $n_{l2}:=n_{l}-n_{l1}$
where $[x]$ is the integer part of $x$ for $x\geq 0$, then
 $\overline{\bm{X}}_{ln_{l1}}$, $\bm{S}_{ln_{l1}}$ and
$\overline{\bm{X}}_{ln_{l2}}$, $\bm{S}_{ln_{l2}}$ stand for the
sample mean vectors and covariances matrices of the first $n_{l1}$
samples and the rest $n_{l2}$ samples, respectively.
We also define the pooled sample covariance denoted by
\begin{eqnarray}
\bm{E}_{1}=\dfrac{1}{n-k}\sum\limits_{l=1}^{k}\sum\limits_{i=1}^{n_{l}}(\bm{X}_{li}-\overline{\bm{X}}_{l}){(\bm{X}_{li}-\overline{\bm{X}}_{l})\trans}
\label{eqn:E1}
\end{eqnarray}
and
\begin{eqnarray}
\bm{E}_{2}=\dfrac{1}{k-1}\sum\limits_{l=1}^{k}n_{l}(\overline{\bm{X}}_{l}-\overline{\bm{X}}){(\overline{\bm{X}}_{l}-\overline{\bm{X}})\trans}.
\label{eqn:E2}
\end{eqnarray}
Finally, we define  $ \widetilde{\bm{\mu}}=\dfrac{1}{n}\sum\limits_{l=1}^{k}n_{l}\bm{\mu}_{l}$
and
$\overline{\bm{\mu}}=\dfrac{1}{k}\sum\limits_{l=1}^{k}\bm{\mu}_{l}$
as weighted mean vector and average mean vector of the population
means $\bm{\mu}_{1},\ldots,\bm{\mu}_{k}$, respectively.

\subsection{Model}
We assume that random samples $\bm{X}_{li}$'s are generated from a
factor model in multivariate analysis which are commonly used in
many existing studies, for example, \cite{Bai:Saranadasa:1996},
\cite{Chen:Qin:2010} and \cite{Hu:Bai:2015}. More formally, some
moment conditions on the distributions of random samples
$\bm{X}_{li}$ are imposed as follows; for every
$i\in\{{1,\ldots,n_{l}}\}$ and $l\in\{{1,\ldots,k\}}$, we consider
\begin{align}\label{eq2}
\bm{X}_{li}=\bm{\Gamma}_{l}\bm{Z}_{li}+\bm{\mu}_{l}
\end{align}
where $\bm{\Gamma}_{l}$ is a $p\times r$ matrix for some $r\geq p$
such that $\bm{\Gamma}_{l}\bm{\Gamma}_{l}\trans=\bm{\Sigma_{l}}$ and
$\{\bm{Z}_{li}\}_{i=1}^{n_{l}}$ are $r-$variate independent and
identically distributed (i.i.d.) random vectors with
${\rm{E}}({\bm{Z}}_{li})=0$ and
${\rm{Var}}({\bm{Z}}_{li})=\bm{I}_{r}$.
Moreover, we assume
${\rm{E}}(z_{lij}^{4})=3+\gamma_{l}<+\infty$ and $z_{lij}$'s are
independent for all $j=1,\ldots,r$; $i=1,\ldots,n_{l}$ and
$l=1,\ldots,k$, where
${\bm{Z}}_{li}=(z_{li1},\ldots,z_{lir})\trans$.

\subsection{Assumptions}
We first state the main conditions which will be used in the proof of asymptotic results of our proposed test.
The three conditions, (A1), (A2) and (A3) are as follows:

\begin{itemize}
\item[(A1)]$\lim\limits_{n\to\infty}n_{l}/n=\lambda_{l}\in(0,1)$ for $l=1,\ldots,k$.\label{A1}
\item[(A2)]${(\bm{\mu}_{l}-\bm{\mu}_{s})\trans}\bm{\Sigma}_{l}(\bm{\mu}_{l}-\bm{\mu}_{s})=o\left\{n^{-1}{\rm{tr}}\left(\sum\limits_{i=1}^{k}\bm{\Sigma}_{i}\right)^{2}\right\}$ for
$l,s\in\{1,\ldots,k\}$ as $n$ and $p\to\infty$.\label{A2}

\item[(A3)]${\rm{tr}}(\bm{\Sigma}_{l}\bm{\Sigma}_{s}\bm{\Sigma}_{i}\bm{\Sigma}_{j})=o\left\{{\rm{tr}}^{2}\left(\sum\limits_{i=1}^{k}\bm{\Sigma}_{i}\right)^{2}\right\}$ for $l,s,i,j\in\{1,\ldots,k\}$ as $p\to\infty$.\label{A3}
\end{itemize}
(\CG{A1}) implies that all sample sizes have the same increasing rate except constant terms.
(\CG{A2}) is used in a local alternative for the power function of the proposed test
and it is actually an extension of (3.3) in \cite{Chen:Qin:2010} to the case of multi-groups ($k\geq 2$).
Similarly, (\CG{A3}) can be
seen as an extension of the condition (3.6) in
\cite{Chen:Qin:2010} to the case of multi-groups.

\section{Main results}\label{sec:3}
In this section we present a new proposed test statistic and its asymptotic
properties  under the conditions (\CG{A1})--(\CG{A3}).

\subsection{The proposed test statistic}
Our proposed test is motivated by \cite{Schott:2007} and
``leave-one-out" idea of \cite{Chen:Qin:2010}. \cite{Schott:2007}
tested the hypothesis \eqref{eq1} under MANOVA based on
\begin{eqnarray}
T_{S}:={\rm{tr}}(\bm{E}_{2})-{\rm{tr}}(\bm{E}_{1}).
\label{TS}
\end{eqnarray}
where $\bm{E}_1$ and $\bm{E}_2$ are defined in \eqref{eqn:E1} and
\eqref{eqn:E2}. The asymptotic normality of $T_{S}$ was derived in
\cite{Schott:2007}, hence a test statistic was formulated by
standardizing $T_{S}$ with an asymptotically ratio-consistent
estimator of its standard deviation. The main assumptions in
\cite{Schott:2007} are as follows:
\begin{itemize}
\item[(A4)]The random samples ${\bm{X}}_{li}$'s come from normal model
$N(\bm{\mu}_{l},\bm{\Sigma})$ for $i=1,\ldots,n_{l}$ and
$l=1,\ldots,k$.\label{A4}
\item[(A5)]$\lim\limits_{n\to\infty}p/n\in(0,1)$.\label{A5}
\item[(A6)]$0<\lim\limits_{p\to\infty}{\rm{tr}}(\bm{\Sigma}^{2i})/p<\infty$ for $i=1$ \text{or} $2$.\label{A6}
\end{itemize}
With (\CG{A4}),  the asymptotic results in \cite{Schott:2007}  were
derived under MANOVA which is the case of homogeneous covariance
matrices under multivariate normality of data.  (\CG{A5}) means that
the sample dimension $p$ and sample size $n$ have the same order and
the total number of samples should be larger than the dimension $p$.
However, our proposed test and \cite{Hu:Bai:2015} need some implicit
relationship between $n$ and $p$ through the condition (\CG{A3})
rather than explicit restriction on $n$ and $p$ as in (\CG{A5}).
Under $\bm{\Sigma}_{1}=\cdots=\bm{\Sigma}_{k}=\bm{\Sigma}$,
(\CG{A6}) is a stronger condition than (\CG{A3}) since (\CG{A3}) is
$\rm{tr}(\Sigma^4) =o(\rm{tr}^2(\bm\Sigma^2))$ showing that
(\CG{A6}) implies (\CG{A3}) through
${\rm{tr}}(\bm{\Sigma}^{4})/{\rm{tr}}^{2}(\bm{\Sigma}^{2})=p^{-1}{\rm{tr}(\bm{\Sigma}^{4})/p}\Big/{\rm{tr}}^{2}(\bm{\Sigma}^{2})/p^{2}=o(1)$
as $p\to\infty$. Thus, considering all these,  it is clear that
(\CG{A4})-(\CG{A6}) are stronger conditions than
(\CG{A1})-(\CG{A3}).

We modify $T_{S}$ in \eqref{TS} by removing the terms $\bm{X}_{li}\trans\bm{X}_{li}$
which is also done in \cite{Chen:Qin:2010}
and get a test
statistic denoted by $T$  as follows:
\begin{eqnarray}
T:=\sum\limits_{l=1}^{k}\frac{n-n_{l}}{n(n_{l}-1)}\sum\limits_{i\neq
j}^{n_{l}}{\bm{X}_{li}\trans}\bm{X}_{lj}-\sum\limits_{l\neq
s}^{k}\frac{n_{l}n_{s}}{n}{\overline{\bm{X}}_{l}\trans}\overline{\bm{X}}_{s}.
\label{eqn:T}
\end{eqnarray}
It is worth pointing out that, for two sample BF problem, the
statistic $T$ is the same as \cite{Chen:Qin:2010} except a constant
factor $n_{1}n_{2}/n$. Elementary derivation shows
\begin{eqnarray}
{\rm{E}}(T)=\sum\limits_{l=1}^{k}n_{l}(\bm{\mu}_{l}-\widetilde{\bm{\mu}})\trans(\bm{\mu}_{l}-\widetilde{\bm{\mu}})
\label{eqn:ET}
\end{eqnarray}
where $\widetilde{\bm{\mu}} = \sum_{l=1}^k \lambda_l \bm{\mu}_l$ for $\lambda_l =n_l/n$.
In \cite{Hu:Bai:2015}, their test statistic is based on a statistic,
say $T_H$, of which the expected value is
${\rm{E}}(T_H)=\sum\limits_{l=1}^{k}(\bm{\mu}_{l}-\overline{\bm{\mu}})\trans(\bm{\mu}_{l}-\overline{\bm{\mu}})$ where $\overline{\bm{\mu}}=\frac{1}{k}\sum_{l=1}^k \bm{\mu}_l$.
The deviation of $\bm{\mu}_l$ from $\widetilde{\bm{\mu}}$ in \eqref{eqn:ET} is
weighted by the corresponding sample size $n_l$ which can emphasize
the deviations of populations with large sample sizes.  On the other
hand, $T_H$ in \cite{Hu:Bai:2015} gives all equal weight to the
deviations of $\bm{\mu}$ from overall mean $\overline{\bm{\mu}}$.
This difference leads to different asymptotic powers of test
statistics based on $T$ and $T_H$.

We now propose a test statistic based on $T$ in \eqref{eqn:T}.
It can be shown that the variance of $T$ is
\begin{eqnarray}
{\rm{Var}}(T)&=&\frac{2}{n^{2}}\left\{\sum\limits_{l=1}^{k}\frac{n_{l}(n-n_{l})^{2}}{n_{l}-1}{\rm{tr}}(\bm{\Sigma}_{l}^{2})
+\sum\limits_{l\neq{s}}^{k}n_{l}n_{s}{\rm{tr}}(\bm{\Sigma}_{l}\bm{\Sigma}_{s})\right\}+4\sum\limits_{l=1}^{k}n_{l}(\bm{\mu}_{l}-\widetilde{\bm{\mu}})\trans\bm{\Sigma}_{l}(\bm{\mu}_{l}-\widetilde{\bm{\mu}}) \nonumber \\
 &=& \sigma_T^2 + 4\sum\limits_{l=1}^{k}n_{l}(\bm{\mu}_{l}-\widetilde{\bm{\mu}})\trans\bm{\Sigma}_{l}(\bm{\mu}_{l}-\widetilde{\bm{\mu}})
 \label{eqn:VarT}
 \end{eqnarray}
where
%From ${\rm{Var}}(T)$,  the variance of $T$ under the $H_0$ denoted by $ \sigma_T^2$  is
\begin{align*}
\sigma_{T}^{2}:=\frac{2}{n^{2}}\left\{\sum\limits_{l=1}^{k}\frac{n_{l}(n-n_{l})^{2}}{n_{l}-1}{\rm{tr}}(\bm{\Sigma}_{l}^{2})
+\sum\limits_{l\neq{s}}^{k}n_{l}n_{s}{\rm{tr}}(\bm{\Sigma}_{l}\bm{\Sigma}_{s})\right\}.
\end{align*}
Note that ${\rm{Var}}(T) = \sigma_T^2$ under $H_0$. From (\CG{A1})
and (\CG{A2}), we have
\begin{align}
(\bm{\mu}_{l}-\widetilde{\bm{\mu}})\trans\bm{\Sigma}_{l}(\bm{\mu}_{l}-\widetilde{\bm{\mu}})=o\left\{n^{-1}{\rm{tr}}\left(\sum\limits_{l=1}^{k}\bm{\Sigma}_{l}\right)^{2}\right\}
\label{eqn:mu}
\end{align}
for $l=1,\ldots,k$ and  by combining ${\rm{Var}}(T)$ and \eqref{eqn:mu}, we obtain
\begin{align*}
{\rm{Var}}(T)=\sigma_{T}^{2}\{1+o(1)\}.
\end{align*}

In order to formulate a test procedure, we should give an
asymptotically ratio-consistent estimator of $\sigma_{T}$. There are
many different estimators proposed in existing studies.  We
adopt two different estimators which are stated in the following two lemmas.

The first one is based on
\cite{Aoshima:Yata:2011} which is given in Lemma \CG{3.1}.
It should
be noted that the requirements for obtaining asymptotically
ratio-consistent estimator of $\sigma_{T}$ in
\cite{Aoshima:Yata:2011} are different from (\CG{A1})-(\CG{A3}). Our
assumption on $\Sigma_{i}$'s in (\CG{A3}) is weaker than those
assumptions (\CG{A-iv} and \CG{A-v}) in \cite{Aoshima:Yata:2011}.

\begin{lemma}\label{lem3.1}
Suppose we have the following estimator
of $\sigma_T^2$
\begin{align*}
\widehat{\sigma_{T}}^{2}:=\frac{2}{n^{2}}\left\{\sum\limits_{l=1}^{k}\frac{n_{l}(n-n_{l})^{2}}{n_{l}-1}\widehat{{\rm{tr}}(\bm{\Sigma}_{l}^{2})}+\sum\limits_{l\neq{s}}^{k}n_{l}n_{s}\widehat{{\rm{tr}}(\bm{\Sigma}_{l}\bm{\Sigma}_{s})}\right\},
\end{align*}
then we have the ratio consistency of $\widehat{\sigma_T}^2$, i.e.,
\begin{align*}
\widehat{\sigma_{T}}/\sigma_{T}\overset{pr}\longrightarrow 1
\end{align*}
where
$\widehat{{\rm{tr}}(\bm{\Sigma}_{l}^{2})}:={\rm{tr}}(\bm{S}_{ln_{l1}}\bm{S}_{ln_{l2}})$
and
$\widehat{{\rm{tr}}(\bm{\Sigma}_{l}\bm{\Sigma}_{s})}:={\rm{tr}}(\bm{S}_{ln_{l}}\bm{S}_{sn_{s}})$
are asymptotically ratio-consistent estimators of
${\rm{tr}}(\bm{\Sigma}_{l}^2)$ and
${\rm{tr}}(\bm{\Sigma}_{l}\bm{\Sigma}_{s})$, respectively, for
$l\neq s$ and $l$, $s=1,\ldots,k$.
\end{lemma}
\begin{proof} See Appendix.
\end{proof}

\bigskip
The other estimator of $\sigma_T$
is the estimator used in \cite{Bai:Saranadasa:1996} and \cite{Hu:Bai:2015}
which is stated in the following lemma.

\begin{lemma}\label{lem3.2}
(\cite{Hu:Bai:2015}) Suppose
\begin{align*}
\widetilde{\sigma_{T}}^{2}:=\frac{2}{n^{2}}\left\{\sum\limits_{l=1}^{k}\frac{n_{l}(n-n_{l})^{2}}{n_{l}-1}\widetilde{{\rm{tr}}(\bm{\Sigma}_{l}^{2})}+\sum\limits_{l\neq{s}}^{k}n_{l}n_{s}\widehat{{\rm{tr}}(\bm{\Sigma}_{l}\bm{\Sigma}_{s})}\right\},
\end{align*}
then
\begin{align*}
\widetilde{\sigma_{T}}/\sigma_{T}\overset{pr}\longrightarrow 1
\end{align*}
where
\begin{align*}
\widetilde{{\rm{tr}}(\bm{\Sigma}_{l}^{2})}=\frac{(n_{l}-1)^2}{(n_{l}+1)(n_{l}-2)}\left\{{\rm{tr}}(\bm{S}_{ln_{l}}^{2})-\frac{1}{n_{l}-1}{\rm{tr}}^{2}(\bm{S}_{ln_{l}})\right\}.
\end{align*}
\end{lemma}

On the basis of Lemmas \CG{3.1} and \CG{3.2}, we propose two test
statistics which are
\begin{eqnarray}
\widehat{T}_{1}:=\frac{T}{\widehat{\sigma_{T}}}~~\mbox{and}~~
\widehat{T}_{2}:=\frac{T}{\widetilde{\sigma_{T}}}.
\label{eqn:proposed}
\end{eqnarray}

In the following section, we prove the asymptotic normality of the
proposed tests in \eqref{eqn:proposed} and their asymptotic power
functions.

\subsection{Asymptotic distributions of the proposed test statistic}
The following theorems establish the asymptotic normality of the new
test statistic \eqref{eqn:proposed} under the $H_0$ and their power
function under the $H_1$, when data dimension $p$ and data size $n$
increase to infinity.

\begin{theorem}\label{theo3.1}
Under (\CG{A1}), (\CG{A3}) and $H_{0}$, as $n, p\to\infty$,
$P(\widehat{T}\geq\xi_{\alpha})=\alpha+o(1)$, where $\xi_{\alpha}$
is the upper $\alpha$ quantile of standard normal distribution where
$\widehat{T}$ is either $\widehat{T}_1$ or $\widehat{T}_2$ in
\eqref{eqn:proposed}.
\end{theorem}
\begin{proof} See Appendix. \end{proof}

The following theorem shows the asymptotic power function of the
proposed test.
\begin{theorem}\label{theo3.2}
Under (\CG{A1})-(\CG{A3}) as $n, p\to\infty$,
the asymptotic power function of $\widehat{T}$ ($\widehat{T}_1$ or $\widehat{T}_2$) is
\begin{align}\label{eq3}
P(\widehat{T}\geq\xi_{\alpha})=\Phi\left\{-\xi_{\alpha}+\frac{\dfrac{\sqrt{2}}{2}n\sum\limits_{l=1}^{k}\lambda_{l}(\bm{\mu}_{l}-\widetilde{\bm{\mu}})^{\tiny{\rm{T}}}(\bm{\mu}_{l}-\widetilde{\bm{\mu}})}{\sqrt{\sum\limits_{l=1}^{k}(1-\lambda_{l})^{2}{\rm{tr}}(\bm{\Sigma}_{l}^{2})+\sum\limits_{l\neq{s}}^{k}\lambda_{l}\lambda_{s}{\rm{tr}}(\bm{\Sigma}_{l}\bm{\Sigma}_{s})}}\right\}+o(1),
\end{align}
where
$\widetilde{\bm{\mu}}=\sum\limits_{l=1}^{k}\lambda_{l}\bm{\mu}_{l}$
and $\Phi(\cdot)$ is the standard normal cumulative distribution
function.
\end{theorem}
\begin{proof} See Appendix.
\end{proof}

\bigskip

Since MANOVA is a special situation of the $k$ sample BF problem, we
have the following two corollaries which are immediate results from
Theorems \CG{3.1} and \CG{3.2}.

\begin{corollary}
If $\bm{\Sigma}_{1}=\cdots=\bm{\Sigma}_{k}$ and assumptions
(\CG{A1}) and (\CG{A3}) hold, under the null hypothesis $H_{0}$, as
$n,p\to\infty$, we get
$P(\widehat{T}\geq\xi_{\alpha})=\alpha+o(1)$.
\end{corollary}

\begin{corollary}
Suppose $\bm{\Sigma}_{1}=\cdots=\bm{\Sigma}_{k}$ and assumptions
(\CG{A1}) and (\CG{A3}) hold. Under the local alternative (\CG{A2}),
as $n,p\to\infty$, we have
\begin{align*}
P(\widehat{T}\geq\xi_{\alpha})=\Phi\left\{-\xi_{\alpha}+\frac{n\sum\limits_{l=1}^{k}\lambda_{l}
(\bm{\mu}_{l}-\widetilde{\bm{\mu}})^{\tiny{\rm{T}}}(\bm{\mu}_{l}-\widetilde{\bm{\mu}})}{\sqrt{2(k-1){\rm{tr}}(\bm{\Sigma}_{1}^{2})}}\right\}+o(1).
\end{align*}
\end{corollary}

\section{Theoretical comparisons and simulations}\label{sec:4}
In this section, we provide theoretical comparisons between the
proposed test and some existing test. For $k$ sample BF problem,
\cite{Cao:2014} and \cite{Hu:Bai:2015} construct test statistics via
the same statistic
\begin{equation*}
T_{CH}=(k-1)\sum\limits_{l=1}^{k}\sum\limits_{i\neq
j}^{n_{l}}{\bm{X}_{li}\trans}\bm{X}_{lj}\big/n_{l}(n_{l}-1)-\sum\limits_{l\neq
s}^{k}{\bm{\overline{X}}_{l}\trans}\bm{\overline{X}}_{s},
\end{equation*}
which is an extension of the two sample test in \cite{Chen:Qin:2010}
to the case of $k$ samples. Depending on different estimators of
variance of $T_{CH}$, different test statistics have been proposed.
\cite{Cao:2014} used two different estimators of variance of
$T_{CH}$. One is similar to that in \cite{Chen:Qin:2010} and the
other is the same as that in Lemma \CG{3.1}. On the other hand,
\cite{Hu:Bai:2015} used the similar estimator to that in
\cite{Bai:Saranadasa:1996}. Under the assumptions similar to
(\CG{A1})-(\CG{A3}), \cite{Cao:2014} and \cite{Hu:Bai:2015} obtained
the same asymptotic distribution of their test statistics, say
$\widehat{T}_{CH} = T_{CH}/\hat \sigma_{CH}$ where $\hat
\sigma_{CH}$ represents the estimators of variance of $T_{CH}$
considered in \cite{Cao:2014} and \cite{Hu:Bai:2015}, as follows:
\begin{align}\label{eq4}
P(\widehat{T}_{CH}\geq\xi_{\alpha})=\Phi\left\{-\xi_{\alpha}+\frac{\dfrac{\sqrt{2}}{2}kn\sum\limits_{l=1}^{k}(\bm{\mu}_{l}-\overline{\bm{\mu}})^{\tiny{\rm{T}}}(\bm{\mu}_{l}-\overline{\bm{\mu}})}{\sqrt{(k-1)^{2}\sum\limits_{l=1}^{k}\lambda_{l}^{-2}{\rm{tr}}(\bm{\Sigma}_{l}^{2})+\sum\limits_{l\neq
s}(\lambda_{l}\lambda_{s})^{-1}{\rm{tr}}(\bm{\Sigma}_{l}\bm{\Sigma}_{s})}}\right\}+o(1).
\end{align}

Since all tests in \cite{Cao:2014} and \cite{Hu:Bai:2015} have the
same asymptotic distribution, we use the test statistic in
\cite{Hu:Bai:2015}, $\widehat{T}_{H}:=T_{CH}/\widetilde{\sigma}$
where
\begin{align*}
\widetilde{\sigma}^{2}=2(k-1)^{2}\sum\limits_{l=1}^{k}\widetilde{{\rm{tr}}(\bm{\Sigma}_{l}^{2})}\big/{n_{l}(n_{l}-1)}+\sum\limits_{l\neq
s}^{k}\widehat{2{\rm{tr}}(\bm{\Sigma}_{l}\bm{\Sigma}_{s})}\big/{n_{l}n_{s}}.
\end{align*}
We provide numerical studies and theoretical comparisons between our
proposed test statistic in \eqref{eqn:proposed}  and $\widehat{T}_H$
in the following sections.

\subsection{Theoretical comparisons}\label{sec:4.1}
We first compare the power functions of the proposed test
$\widehat{T}$ and $\widehat{T}_H$ when all sample sizes are the
same, where $\widehat{T}$ is either $\widehat{T}_1$ or
$\widehat{T}_2$ in \eqref{eqn:proposed}. The following Corollary
\CG{4.1} states that $\widehat{T}$ and $\widehat{T}_{H}$ have the
same asymptotic power under balanced model. This can be shown
directly from \eqref{eq3} and \eqref{eq4}.
\begin{corollary}\label{cor:balanced}
The test statistics $\widehat{T}$ and $\widehat{T}_{H}$ have the
same asymptotic power under balanced model which means each group
has equal sample sizes.
\end{corollary}

For more general cases such as unbalanced sample sizes, it is not
easy to compare the asymptotic power functions of $\widehat{T}$ and
$\widehat{T}_{H}$. We compare all test statistics under simple and
typical situations so that we can compare the power functions
analytically. To obtain rough depiction, we assume
$\bm{\Sigma}_{1}=\cdots=\bm{\Sigma}_{k}$ for the following cases. We
define the asymptotic relative efficiency (ARE) of $\widehat{T}$ to
$\widehat{T}_{H}$ which is the ratio of two signal-to-noise ratios:
\begin{align}
{\rm{ARE}}(\widehat{T},\widehat{T}_{H})&:=\frac{{\rm{E}}(\widehat{T})}{\sqrt{{\rm{Var}}(\widehat{T})}}\Biggr/\frac{{\rm{E}}(\widehat{T}_H)}{\sqrt{{\rm{Var}}(\widehat{T}_H)}}\nonumber\\
&=\frac{n\sum\limits_{l=1}^{k}\lambda_{l}(\bm{\mu}_{l}-\widetilde{\bm{\mu}})^{\trans}(\bm{\mu}_{l}-\widetilde{\bm{\mu}})}{\sqrt{2(k-1){\rm{tr}}(\bm{\Sigma}_{1}^2)}}\Biggr/\frac{kn\sum\limits_{l=1}^{k}(\bm{\mu}_{l}-\overline{\bm{\mu}})^{\trans}(\bm{\mu}_{l}-\overline{\bm{\mu}})}{\sqrt{2{\rm{tr}}(\bm{\Sigma}_{1}^{2})}\sqrt{(k-1)^{2}\sum\limits_{l=1}^{k}\lambda_{l}^{-2}+\sum\limits_{l\neq
s}(\lambda_{l}\lambda_{s})^{-1}}}.\label{eqn:ARE}
\end{align}

If the ${\rm{ARE}}(\widehat{T},\widehat{T}_{H})>1$, the asymptotic
power of $\widehat{T}$ is larger than that of $\widehat{T}_H$ from
\eqref{eq3} and \eqref{eq4}.

Based on the ARE \eqref{eqn:ARE}, we consider the following two representative cases:
\begin{enumerate}
\item [$(i)$]  $\bm{\mu}_{1}=\cdots=\bm{\mu}_{k-1}\neq \bm{\mu}_{k}$.
Without loss of generality, we set
$\bm{\mu}_{1}=\cdots=\bm{\mu}_{k-1}=0\neq \bm{\mu}_{k}$. Then, we have
\begin{align*}
{\rm{ARE}}(\widehat{T},\widehat{T}_{H})&
=\frac{\lambda_{k}(1-\lambda_{k})}{(k-1)\sqrt{k-1}}\sqrt{k(k-2)\sum\limits_{l=1}^{k}\lambda_{l}^{-2}+\left(\sum\limits_{l=1}^{k}\lambda_{l}^{-1}\right)^{2}}.
\end{align*}
From this, we see that  ${\rm{ARE}}(\widehat{T},\widehat{T}_{H})$ is
larger than 1 if there exists at least one $l\in\{1,\ldots,k-1\}$ such
that $\lambda_{l}$ is very small, for example $\lambda_l$ is close
to 0 for some $l \in \{1,\ldots, k-1 \}$. This is because the right hand side of $ARE(\hat T, \hat T_h)$ is unbounded as $\lambda_l$ is close to $0$ for at least one  $l \in \{1,\ldots, k-1  \}1$.
Since $\sum_{l=1}^k \lambda_l$,
it indicates that if $\lambda_{k}$ is close to 1, then most of $\lambda_{l}$s for $1\leq l \leq k-1$ are close to 0 which results in ${\rm{ARE}}(\widehat{T},\widehat{T}_{H})>1$.
Furthermore, we can get an another low bound of  ${\rm ARE}(\widehat{T}, \widehat{T}_H)$
such as  ${\rm ARE}(\widehat{T},\widehat{T}_{H}) >
k^{2}\lambda_{k}(1-\lambda_{k})/(k-1)$ based on mean value and
Jensen's inequality (\cite{Mitrinovic:1993}).
The low bound depends on only $\lambda_k$ and it shows that
if
$\lambda_{k}\in(1/k,(k-1)/k)$, then we have
$k^{2}\lambda_{k}(1-\lambda_{k})/(k-1) > 1$ regardless of configurations of all other $\lambda_l$
for $1\leq l \leq k-1$. This shows that as the number of groups ($k$) increases,
the interval $(1/k, (1-k)/k)$ is getting wider, so
$\widehat{T}$ is expected to have more power than $\widehat{T}_H$ as the number of groups ($k$) increases.

\item [$(ii)$] As a second case,
we assume all mean vectors have the same direction such that  $(\bm{\mu}_l -\bm{\mu}_1) = \tau_l  (\bm{\mu}_k - \bm{\mu}_1) $
for $2\leq l \leq k-1$ and some constants $\tau_l$.
For simplicity, we consider $k=3$ and
$\lambda_{1}=\lambda_{2}$.
For $k=3$, without loss of generality, we can assume
$\bm{\mu}_{3}\neq 0$, $\bm{\mu}_{1}=0$ and
$\bm{\mu}_{2}=\tau\bm{\mu}_{3}$ with $\tau\neq 0$.
From  $ \lambda_1
= \lambda_2 = (1-\lambda_{3})/2$ and , we have
\begin{align*}
{\rm{ARE}}(\widehat{T},\widehat{T}_{H})&=\frac{\tau^{2}\lambda_{3}^{-1}+(\tau-2)^{2}}{4\sqrt{2}(\tau^{2}-\tau+1)}\sqrt{9\lambda_{3}^{2}+1}.
\end{align*}
For all $\tau (\neq 0)$, the equation
${\rm{ARE}}(\widehat{T},\widehat{T}_{H})=1$ has a fixed solution
$\lambda_{3}=1/3$.
In addition to this solution, there are more solutions and we provide approximate solutions by
numerical studies as follows:

\begin{itemize}
\item[(a)] The case of $\tau\in(0.009,0.732]\cup [-2.732,-0.009)$: In this case, there are two
solutions of ${\rm{ARE}}(\widehat{T},\widehat{T}_{H})=1$ which are
$(1/3,1)$ and $(\lambda_{3}^{0},1)$, where $0<\lambda_{3}^{0}<1/3$.
Note that $\lambda_{3}^{0}$ depends on $\tau$. If
$\lambda_{3}\in(\lambda_{3}^{0},1/3)$,
${\rm{ARE}}(\widehat{T},\widehat{T}_{H})<1$; otherwise
${\rm{ARE}}(\widehat{T},\widehat{T}_{H})\geq 1$.
This implies that
$\widehat{T}$ has more powers than $\widehat{T}_{H}$ when
$\lambda_{3} \in [\lambda_{3}^{0},1/3]^c$. The left panel
in Figure 1 shows this case.

\item[(b)] The case of $\tau\in(0.724,1.119]\cup (9.353,\infty)\cup (-\infty,-2.732)$: The solutions of
${\rm{ARE}}(\widehat{T},\widehat{T}_{H})=1$ are $(1/3,1)$ and
$(\lambda_{3}^{1},1)$ with $\lambda_{3}^{1}>1/3$. The right panel in
Figure 1 shows that $\lambda_{3}^{1}$ reaches 1 very rapidly when
$\tau > 0.724$ and $\tau\neq 2$. We see that
${\rm{ARE}}(\widehat{T},\widehat{T}_{H})$ is significantly larger
than 1 when $\lambda_{3} < 1/3$ while
${\rm{ARE}}(\widehat{T},\widehat{T}_{H})$ is slightly less than 1
when $1/3<\lambda_{3}<\lambda_{3}^1$.
This shows that $\widehat T$ has significantly larger powers than $\widehat{T}_H$
in most cases while $\widehat{T}_H$ can have slightly more powers;on the other hand,
even when $\widehat{T}_H$ has more powers than $\widehat{T}$, the difference is not that significant.
This case is shown in the
right panel in Figure 1.

\item[(c)] The case of $\tau\in(1.119,9.353]$: There is only one solution of the equation
${\rm{ARE}}(\widehat{T},\widehat{T}_{H})=1$ which is $(1/3,1)$. When
$\lambda_{3}$ is less than $1/3$,
${\rm{ARE}}(\widehat{T},\widehat{T}_{H})>1$ which means that
$\widehat{T}$ has larger  powers than those of
$\widehat{T}_{H}$. Moreover, we see that
${\rm{ARE}}(\widehat{T},\widehat{T}_{H})$ is increasing as
$\lambda_3$ decreases. This case is shown in the left panel
in Figure 2.

\item[(d)] The case of $\tau\in(0,0.009]\cup [-0.009,0)$: The equation
${\rm{ARE}}(\widehat{T},\widehat{T}_{H})=1$ has only one solution
$(1/3,1)$. When $\lambda_{3}$ is more than $1/3$,
${\rm{ARE}}(\widehat{T},\widehat{T}_{H})>1$ which illustrates
$\widehat{T}$ has larger  powers than those of
$\widehat{T}_{H}$. See the right panel in Figure 2 for this case.
\end{itemize}
\end{enumerate}

To summarize,
we see that the proposed test $\widehat{T}$
has potential to have more power than $\widehat{T}_H$ when sample sizes are highly unbalanced
while $\widehat T$ and $\widehat{T}_H$
have the same asymptotic power from Corollary \ref{cor:balanced}.
We provide numerical studies to demonstrate this point in the following section.

\begin{figure}[H]
\centering{ \resizebox{7cm}{6cm}
{\includegraphics[width=90mm]{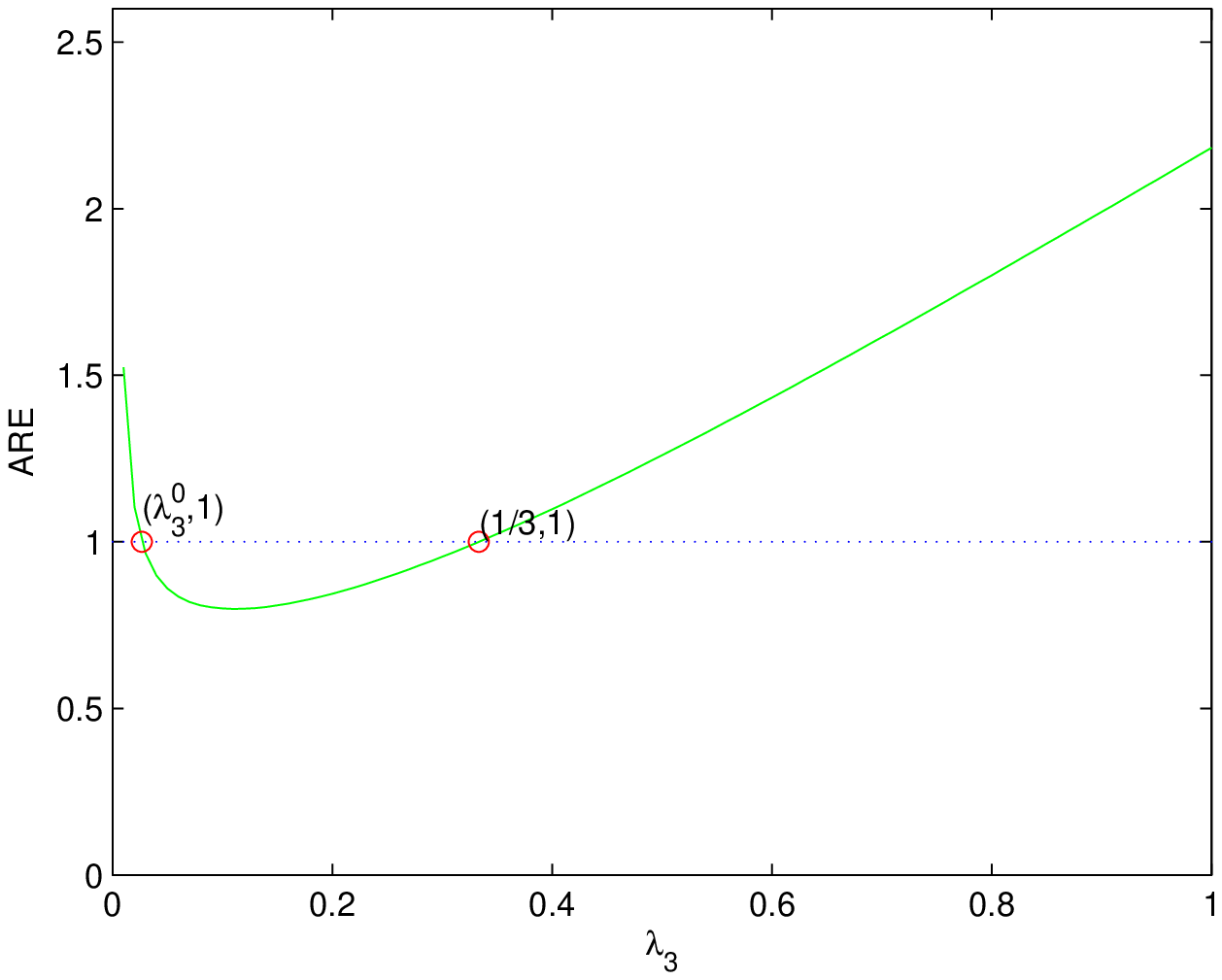}}
 \resizebox{7cm}{6cm}
 {\includegraphics[width=90mm]{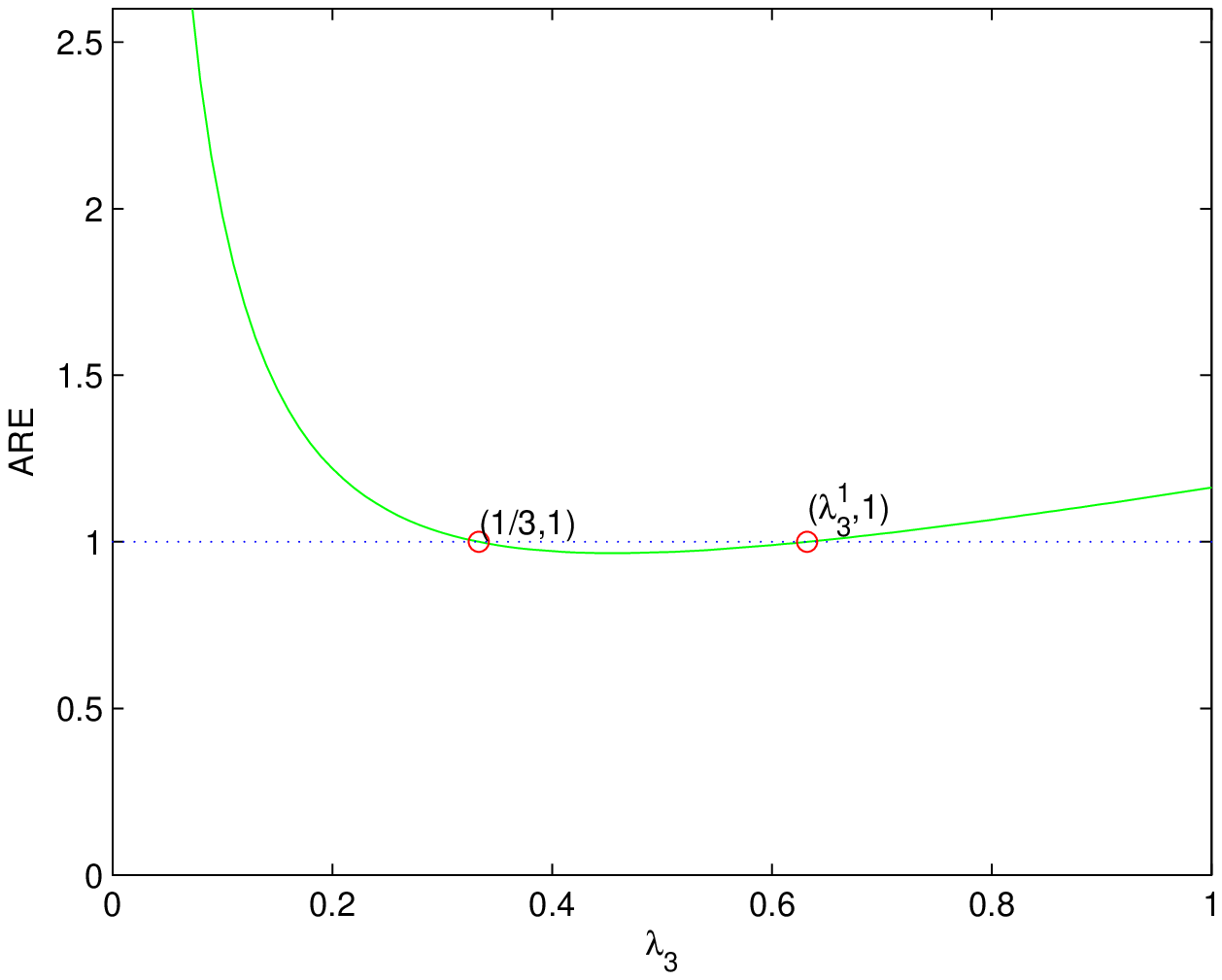}}}{\par{\footnotesize\textbf{Fig. 1} ${\rm{ARE}}(\widehat{T},\widehat{T}_{H})$ for
$\tau=0.2$ (the left) and $\tau=-25$ (the right)}}.
\end{figure}

\begin{figure}[H]
\centering{ \resizebox{7cm}{6cm}
{\includegraphics[width=90mm]{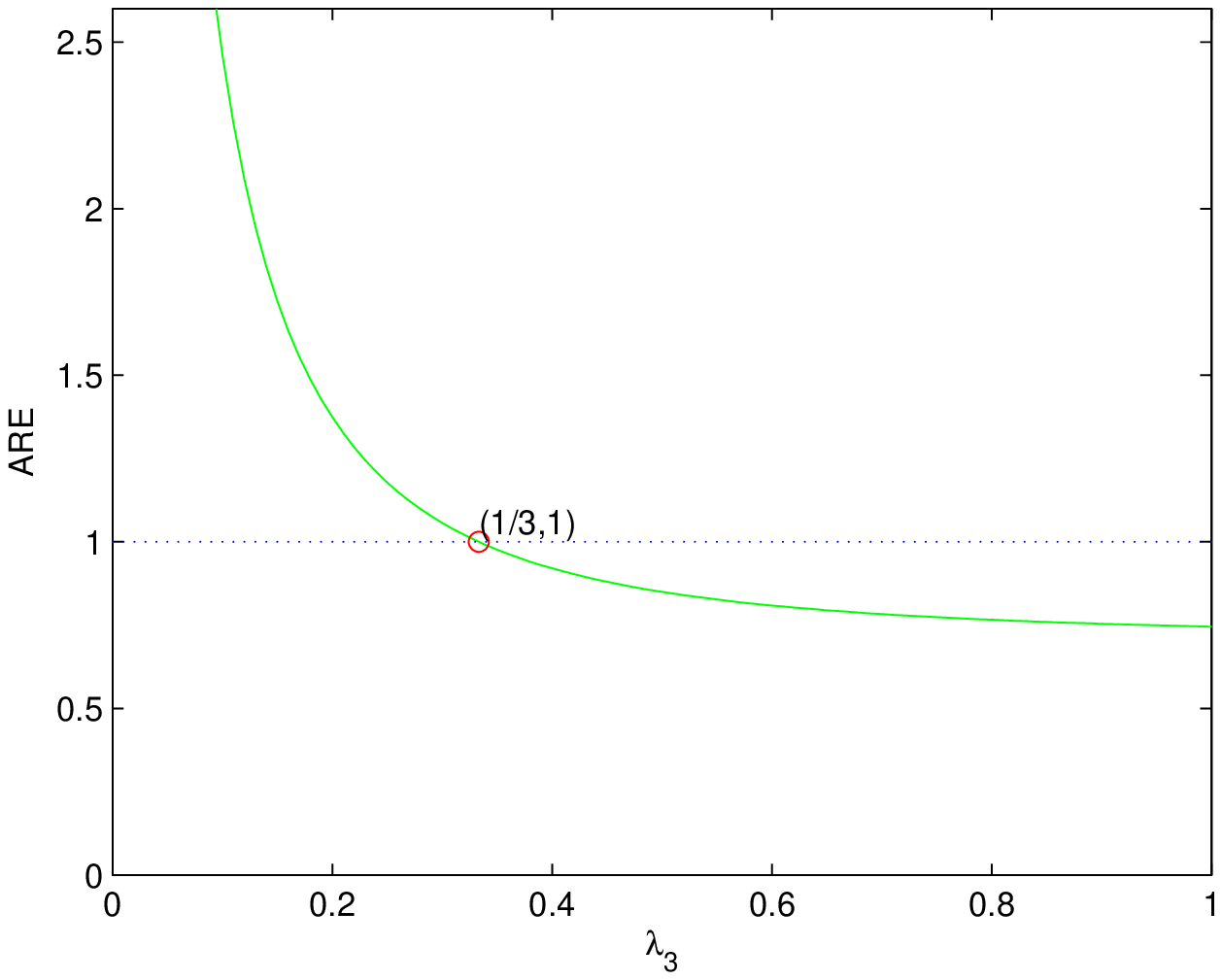}}
 \resizebox{7cm}{6cm}
 {\includegraphics[width=90mm]{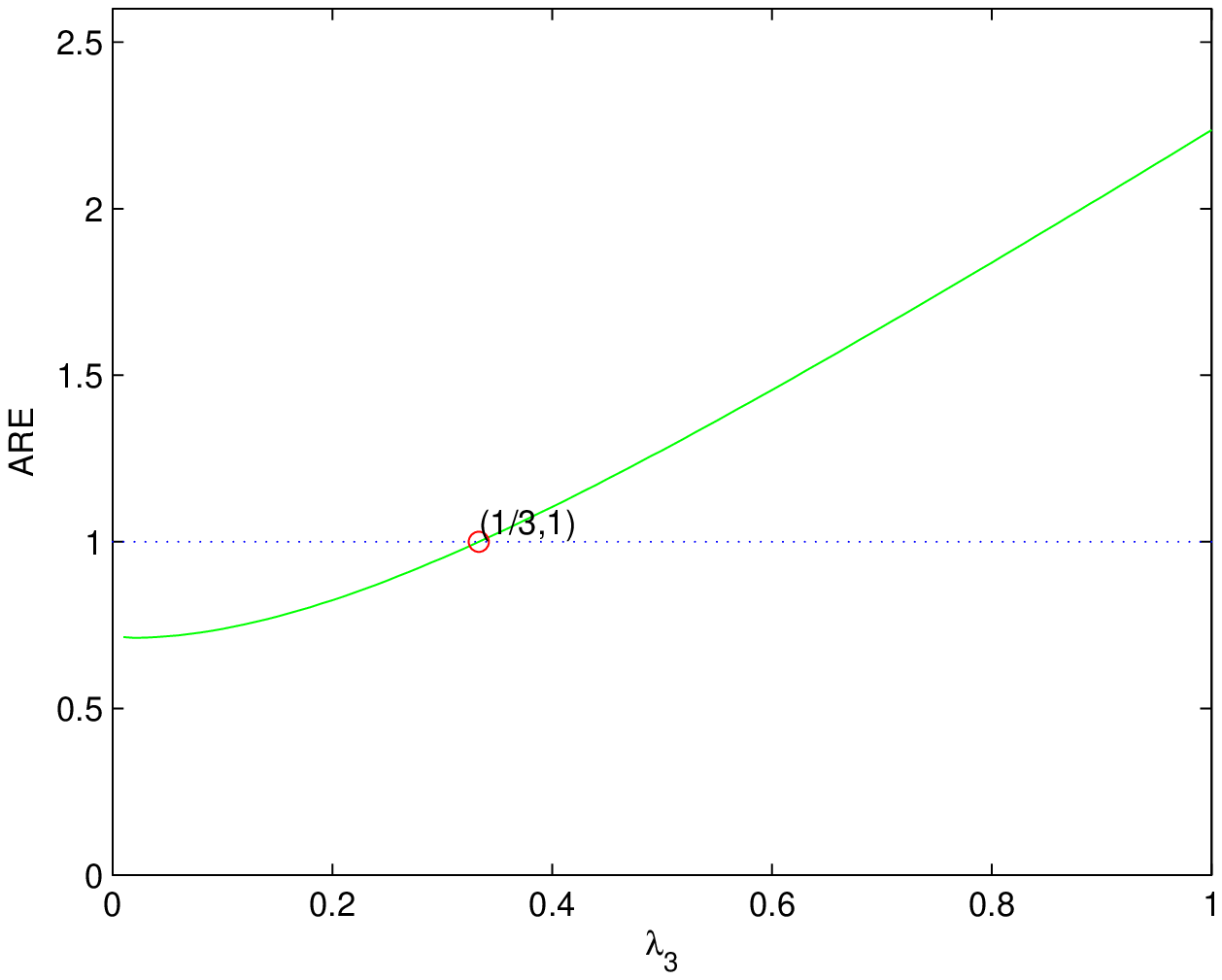}}}{\par{\footnotesize\textbf{Fig. 2}
${\rm{ARE}}(\widehat{T},\widehat{T}_{H})$ for $\tau=2$ (the left)
and $\tau=0.002$ (the right)}}.
\end{figure}

\subsection{Simulations}
As shown in Corollary \ref{cor:balanced}, $\widehat{T}$ and
$\widehat{T}_H$ have the same asymptotic power function for balanced
sample sizes. Therefore, we conduct simulations only for unbalanced
sample sizes to compare $\widehat{T}$ ($\widehat{T}_1$ and
$\widehat{T}_2$) with $\widehat{T}_{H}$. We set $k=3$ and generate
$X_{lij}$ from the following two models.
\begin{itemize}
\item The first model: we consider ``Two-dependence'' moving average model
\begin{align*}
X_{lij}=\rho_{l1}Z_{lij}+\rho_{l2}Z_{l,i,j+1}+\rho_{l3}Z_{l,i,j+2}+\mu_{lj}
\end{align*}
for $i=1,\ldots,n_{l}$, $l=1,2,3$ and $j=1,\ldots,p$, where $Z$'s
are i.i.d. random variables distributed with centered $\chi^{2}(4)$
and $N(0,1)$, respectively. $\rho$'s and $\mu$'s are constants such
that $\bm{\mu}_{l}=(\mu_{l1},\ldots,\mu_{lp})\trans$. Moreover,
$\rho$'s were generated independently from $U(2,3)$ with
$\rho_{11}=2.1984$, $\rho_{12}=2.5743$, $\rho_{13}=2.1316$,
$\rho_{21}=2.8147$, $\rho_{22}=2.9058$, $\rho_{23}=2.1270$,
$\rho_{31}=2.9134$, $\rho_{32}=2.6324$ and $\rho_{33}=2.0975$, and
were kept fixed throughout the simulations. For power studies,
population means are fixed as $\bm{\mu}_{1}=\bm{\mu}_{2}=0$, while
the third mean vector consists of $[0.05*p]$ components equal to
$\delta$ and the others equal to zero where $\delta$ is related to
the following standard parameter
\begin{align}
\theta=\frac{\sum\limits_{l=1}^{3}(\bm{\mu}_{l}-\overline{\bm{\mu}})\trans(\bm{\mu}_{l}-\overline{\bm{\mu}})}{\sqrt{4\sum\limits_{l=1}^{3}\lambda_{l}^{-2}{\rm{tr}}(\bm{\Sigma}_{l}^{2})+\sum\limits_{l\neq
s}^{3}(\lambda_{l}\lambda_{s})^{-1}{\rm{tr}}(\bm{\Sigma}_{l}\bm{\Sigma}_{s})}}.
\label{eqn:theta}
\end{align}
\end{itemize}

\begin{itemize}
\item The second model: for every $i\in\{{1,\ldots,n_{l}}\}$ and
$l\in\{{1,2,3\}}$, we consider
\begin{align*}
\bm{X}_{li}=\bm{\Gamma}_{l}\bm{Z}_{li}+\bm{\mu}_{l},
\end{align*}
where
$\bm{\Sigma}_{l}=\bm{\Gamma}_{l}^2=\bm{W}_{l}\bm{\Psi}_{l}\bm{W}_{l}$,
$\bm{W}_{l}={\rm{diag}}(w_{l1},\ldots,w_{lp})$ with
$w_{lj}=l-(j-1)/p$, $j=1,\ldots,p$, and $\bm{\Psi}_{l}=(\psi_{ljk})$
with $\psi_{ljj}=1$,
$\psi_{ljk}={(-1)^{j+k}}(0.05*b_{l})^{{\vert{j-k}\vert}^{0.1}}$ when
$j\neq{k}$, where $b_{1}=2$, $b_{2}=1$ and $b_{3}=3$. In power
simulation, we set $\bm{\mu}_{1}=0$,
$-\bm{\mu}_{3}=\bm{\mu}_{2}=(u_{1},\ldots,u_{p})^{\trans}$, where
$u_{i}=(-1)^{i}v_{i}$ with $v_{i}'s$ are i.i.d. $U(0,a)$ which
denotes uniform distribution with the support $(0,a)$.
\end{itemize}

We consider three cases for data dimensions and sample sizes.

\begin{itemize}
\item \bm{{\rm{Case~1}}} Let $p$=50, 100, 200, 400, 800, 1000 and $n_{1}=2m$, $n_{2}=3m$ and $n_{3}=5m$ with $m\in\{5, 10,
20\}$.
\end{itemize}

\begin{itemize}
\item \bm{{\rm{Case~2}}} Let $p$=400, $n_{1}=n_{2}$=10, $n_{3}=80$ and $n_{1}=n_{2}=15, n_{3}=70$, respectively.
\end{itemize}

\begin{itemize}
\item \bm{{\rm{Case~3}}} Let $p$=400, $n_{1}$=10, $n_{2}$=20, $n_{3}=70$ and $n_{1}=10, n_{2}=30, n_{3}=60$, respectively.
\end{itemize}

Empirical sizes and powers are computed under the nominal level $\alpha=0.05$
with $5,000$ replications.
For the first model, the standard parameter
$\theta$ in \eqref{eqn:theta} is selected as 0 and 0.005 for size and power, respectively for
\bm{{\rm{Case~1}}}. For the second
model, $a$ is taken as 0 and 0.2 for size and power, respectively in
\bm{{\rm{Case~1}}}. For the two models, we
also consider the cases of  $H_1$ for different configurations of $\theta$ and $a$ in
\bm{{\rm{Case~2}}} and \bm{{\rm{Case~3}}}.
%Throughout the simulations, the results for the normal case have the similar
%pattern with those for the non-normal case, and are not reported
%here.

%\newpage
\begin{table}[H]
\tabcolsep 4mm\par{\footnotesize\textbf{Table 1} Empirical sizes and
powers of $\widehat{T}_1$, $\widehat{T}_2$ and $\widehat{T}_{H}$ in
\bm{{\rm{Case~1}}} under the first model.}
\begin{center}
\def\temptablewidth{1\textwidth}
\begin{tabular*}{\temptablewidth}{@{\extracolsep{\fill}}llllllll}
\hline
&  & \multicolumn{3}{c}{sizes}& \multicolumn{3}{c}{powers}\\
\cline{3-5}    \cline{6-8}
$p$ &  $n$   & $\widehat{T}_1$ & $\widehat{T}_2$ &$\widehat{T}_{H}$   & $\widehat{T}_1$ & $\widehat{T}_2$ & $\widehat{T}_{H}$\\
\hline
50  &50  & 0.0626  & 0.0594  & 0.0600 & 0.0896 & 0.0850 & 0.0808 \\

    &100 & 0.0564  & 0.0534  & 0.0584 & 0.0830 & 0.0810 & 0.0702 \\

    &200 & 0.0576  & 0.0570  & 0.0590 & 0.0830 & 0.0806 & 0.0760 \\

100 &50  & 0.0588  & 0.0546  & 0.0564 & 0.1218 & 0.1176 & 0.0936 \\

    &100 & 0.0630  & 0.0598  & 0.0494 & 0.1204 & 0.1178 & 0.1022 \\

    &200 & 0.0550  & 0.0544  & 0.0560 & 0.1282 & 0.1252 & 0.1054 \\

200 &50  & 0.0590  & 0.0562  & 0.0604 & 0.2190 & 0.2060 & 0.1604 \\

    &100 & 0.0528  & 0.0502  & 0.0548 & 0.2240 & 0.2150 & 0.1618\\

    &200 & 0.0558  & 0.0552  & 0.0536 & 0.2224 & 0.2172 & 0.1664 \\

400 &50  & 0.0536  & 0.0492  & 0.0510 & 0.4848 & 0.4656 & 0.3394 \\

    &100 & 0.0612  & 0.0580  & 0.0604 & 0.4624 & 0.4536 & 0.3412 \\

    &200 & 0.0578  & 0.0556  & 0.0576 & 0.4870 & 0.4812 & 0.3524 \\

800 &50  & 0.0580  & 0.0522  & 0.0512 & 0.9146 & 0.9060 & 0.7688 \\

    &100 & 0.0600  & 0.0560  & 0.0572 & 0.9080 & 0.9040 & 0.7634 \\

    &200 & 0.0534  & 0.0514  & 0.0516 & 0.9106 & 0.9078 & 0.7778\\

1000 &50  & 0.0552  & 0.0488  & 0.0478 & 0.9766 & 0.9732 & 0.8940\\

     &100 & 0.0524  & 0.0500  & 0.0462 & 0.9776 & 0.9764 & 0.8994 \\

     &200 & 0.0550  & 0.0532  & 0.0534 & 0.9752 & 0.9746 & 0.9060 \\
\hline
 \end{tabular*}
\end{center}
\end{table}

\begin{table}[H]
\tabcolsep 4mm\par{\footnotesize\textbf{Table 2} Empirical sizes and
powers of $\widehat{T}_1$, $\widehat{T}_2$ and $\widehat{T}_{H}$ in
\bm{{\rm{Case~1}}} under the second model.}
\begin{center}
\def\temptablewidth{1\textwidth}
\begin{tabular*}{\temptablewidth}{@{\extracolsep{\fill}}llllllll}
\hline
&  & \multicolumn{3}{c}{sizes}& \multicolumn{3}{c}{powers}\\
\cline{3-5}    \cline{6-8}
$p$ &  $n$  & $\widehat{T}_1$ & $\widehat{T}_2$ &$\widehat{T}_{H}$   & $\widehat{T}_1$ & $\widehat{T}_2$ & $\widehat{T}_{H}$\\
\hline
50  &55  & 0.0617  & 0.0557  & 0.0560 & 0.1782 & 0.1697 & 0.1573 \\

    &110 & 0.0602  & 0.0578  & 0.0572 & 0.2828 & 0.2774 & 0.2588 \\

    &220 & 0.0608  & 0.0594  & 0.0598 & 0.5153 & 0.5119 & 0.4848 \\

100 &55  & 0.0634  & 0.0608  & 0.0620 & 0.2276 & 0.2182 & 0.2023 \\

    &110 & 0.0660  & 0.0640  & 0.0658 & 0.4077 & 0.4002 & 0.3698 \\

    &220 & 0.0628  & 0.0622  & 0.0640 & 0.7933 & 0.7909 & 0.7705 \\

200 &55  & 0.0588  & 0.0550  & 0.0564 & 0.2838 & 0.2742 & 0.2556 \\

    &110 & 0.0614  & 0.0604  & 0.0604 & 0.5548 & 0.5484 & 0.5202\\

    &220 & 0.0598  & 0.0588  & 0.0588 & 0.8710 & 0.8692 & 0.8566 \\

400 &55  & 0.0626  & 0.0610  & 0.0612 & 0.3494 & 0.3406 & 0.3226 \\

    &110 & 0.0644  & 0.0630  & 0.0636 & 0.6632 & 0.6596 & 0.6360 \\

    &220 & 0.0586  & 0.0580  & 0.0608 & 0.9424 & 0.9418 & 0.9328 \\

800 &55  & 0.0667  & 0.0658  & 0.0652 & 0.4122 & 0.4046 & 0.3780 \\

    &110 & 0.0628 & 0.0614   & 0.0610 & 0.7762 & 0.7734 & 0.7512\\

    &220 & 0.0562  & 0.0554  & 0.0564 & 0.9928 & 0.9926 & 0.9912\\

1000 &55  & 0.0574  & 0.0564  & 0.0554 & 0.4592 & 0.4516 & 0.4236 \\

    &110 & 0.0588 & 0.0578   & 0.0594 & 0.8406 & 0.8380 & 0.8212\\

    &220 & 0.0566  & 0.0558  & 0.0571 & 0.9952 & 0.9952 & 0.9950\\
\hline
 \end{tabular*}
\end{center}
\end{table}

\begin{table}[H]
\tabcolsep 4mm\par{\footnotesize\textbf{Table 3} Empirical sizes and
powers of $\widehat{T}_1$, $\widehat{T}_2$ and $\widehat{T}_{H}$ in
\bm{{\rm{Case~2}}} under the first model.}
\begin{center}
\def\temptablewidth{1\textwidth}
\begin{tabular*}{\temptablewidth}{@{\extracolsep{\fill}}lllllll}
\hline
 & \multicolumn{6}{c}{$p=400$, $n=100$}\\
\cline{2-7} & \multicolumn{3}{c}{$n_{1}=n_{2}=10, n_{3}=80$}&
\multicolumn{3}{c}{$n_{1}=n_{2}=15, n_{3}=70$}\\
\cline{2-4}    \cline{5-7}
$\theta$  & $\widehat{T}_1$ & $\widehat{T}_2$ &$\widehat{T}_{H}$   & $\widehat{T}_1$ & $\widehat{T}_2$ & $\widehat{T}_{H}$\\
\hline
0     & 0.0572  & 0.0494  & 0.0500  & 0.0606  & 0.0564 & 0.0584 \\

0.002 & 0.0840  & 0.0752  & 0.0670  & 0.0886  & 0.0814 & 0.0678 \\

0.003 & 0.1408  & 0.1234  & 0.0850  & 0.1400 & 0.1304  & 0.0936 \\

0.004 & 0.2358  & 0.2146  & 0.1242  & 0.2444 & 0.2350  & 0.1430 \\

0.005 & 0.3990  & 0.3768  & 0.1884  & 0.4172 & 0.3994  & 0.2248 \\

0.006 & 0.6042  & 0.5786  & 0.2630  & 0.6506 & 0.6316  & 0.3414 \\

0.007 & 0.8098  & 0.7948  & 0.4140  & 0.8340 & 0.8242  & 0.5066 \\

0.008 & 0.9324  & 0.9232  & 0.5828  & 0.9442 & 0.9390  & 0.6870 \\

0.009 & 0.9840  & 0.9816  & 0.7376  & 0.9898 & 0.9882  & 0.8414 \\
\hline
 \end{tabular*}
\end{center}
\end{table}

\begin{table}[H]
\tabcolsep 4mm\par{\footnotesize\textbf{Table 4} Empirical sizes and
powers of $\widehat{T}_1$, $\widehat{T}_2$ and $\widehat{T}_{H}$ in
\bm{{\rm{Case~2}}} under the second model.}
\begin{center}
\def\temptablewidth{1\textwidth}
\begin{tabular*}{\temptablewidth}{@{\extracolsep{\fill}}lllllll}
\hline
 & \multicolumn{6}{c}{$p=400$, $n=100$}\\
\cline{2-7} & \multicolumn{3}{c}{$n_{1}=n_{2}=10, n_{3}=80$}&
\multicolumn{3}{c}{$n_{1}=n_{2}=15, n_{3}=70$}\\
\cline{2-4}    \cline{5-7}
$a$   & $\widehat{T}_1$ & $\widehat{T}_2$ &$\widehat{T}_{H}$   & $\widehat{T}_1$ & $\widehat{T}_2$ & $\widehat{T}_{H}$\\
\hline
0    & 0.0620  & 0.0560  & 0.0546  & 0.0546 & 0.0508  & 0.0516 \\

0.05 & 0.0844  & 0.0756  & 0.0690  & 0.0816  & 0.0774  & 0.0732 \\

0.10 & 0.1794  & 0.1640  & 0.1314  & 0.1540 & 0.1468 & 0.1380 \\

0.15 & 0.3424  & 0.3254  & 0.2600  & 0.3812 & 0.3718 & 0.3456 \\

0.20 & 0.5314  & 0.5110  & 0.4190  & 0.6314 & 0.6188 & 0.6086 \\

0.25 & 0.8102  & 0.7954  & 0.7182  & 0.8066 & 0.7996 & 0.7902 \\

0.30 & 0.9276  & 0.9178  & 0.8710  & 0.9646 & 0.9626 & 0.9552 \\
\hline
 \end{tabular*}
\end{center}
\end{table}

\begin{table}[H]
\tabcolsep 4mm\par{\footnotesize\textbf{Table 5} Empirical sizes and
powers of $\widehat{T}_1$, $\widehat{T}_2$ and $\widehat{T}_{H}$ in
\bm{{\rm{Case~3}}} under the first model.}
\begin{center}
\def\temptablewidth{1\textwidth}
\begin{tabular*}{\temptablewidth}{@{\extracolsep{\fill}}lllllll}
\hline
 & \multicolumn{6}{c}{$p=400$, $n=100$}\\
\cline{2-7} & \multicolumn{3}{c}{$n_{1}=10, n_{2}=20, n_{3}=70$}&
\multicolumn{3}{c}{$n_{1}=10, n_{2}=30, n_{3}=60$}\\
\cline{2-4}    \cline{5-7}
$\theta$  & $\widehat{T}_1$ & $\widehat{T}_2$ &$\widehat{T}_{H}$   & $\widehat{T}_1$ & $\widehat{T}_2$ & $\widehat{T}_{H}$\\
\hline
0     & 0.0604  & 0.0507  & 0.0512  & 0.0538  & 0.0440 & 0.0508 \\

0.002 & 0.0944  & 0.0826  & 0.0670  & 0.0974  & 0.0834 & 0.0698 \\

0.003 & 0.1572  & 0.1402  & 0.0916  & 0.1850 & 0.1590  & 0.1034 \\

0.004 & 0.2736  & 0.2510  & 0.1314  & 0.3290 & 0.2928  & 0.1476 \\

0.005 & 0.4598  & 0.4318  & 0.2098  & 0.5628 & 0.5282  & 0.2548 \\

0.006 & 0.6736  & 0.6462  & 0.3144  & 0.7990 & 0.7714  & 0.4034 \\

0.007 & 0.8716  & 0.8554  & 0.4910  & 0.9458 & 0.9364  & 0.5934 \\

0.008 & 0.9586  & 0.9526  & 0.6704  & 0.9902 & 0.9874  & 0.7566 \\

0.009 & 0.9934  & 0.9922  & 0.8254  & 0.9996 & 0.9996  & 0.9002 \\
\hline
 \end{tabular*}
\end{center}
\end{table}

\begin{table}[H]
\tabcolsep 4mm\par{\footnotesize\textbf{Table 6} Empirical sizes and
powers of $\widehat{T}_1$, $\widehat{T}_2$ and $\widehat{T}_{H}$ in
\bm{{\rm{Case~3}}} under the second model.}
\begin{center}
\def\temptablewidth{1\textwidth}
\begin{tabular*}{\temptablewidth}{@{\extracolsep{\fill}}lllllll}
\hline
 & \multicolumn{6}{c}{$p=400$, $n=100$}\\
\cline{2-7} & \multicolumn{3}{c}{$n_{1}=10, n_{2}=20, n_{3}=70$}&
\multicolumn{3}{c}{$n_{1}=10, n_{2}=30, n_{3}=60$}\\
\cline{2-4}    \cline{5-7}
$a$   & $\widehat{T}_1$ & $\widehat{T}_2$ &$\widehat{T}_{H}$   & $\widehat{T}_1$ & $\widehat{T}_2$ & $\widehat{T}_{H}$\\
\hline
0     & 0.0554  & 0.0526  & 0.0516  & 0.0608 & 0.0588  & 0.0568 \\

0.05 & 0.0836  & 0.0792  & 0.0762  & 0.0954  & 0.0942  & 0.0876 \\

0.10 & 0.2226  & 0.2164  & 0.1928  & 0.2160 & 0.2132 & 0.1820 \\

0.15 & 0.4210  & 0.4136  & 0.3666  & 0.4394 & 0.4352 & 0.3818 \\

0.20 & 0.6726  & 0.6672  & 0.6236  & 0.6872 & 0.6816 & 0.6288 \\

0.25 & 0.8964  & 0.8926  & 0.8660  & 0.9110 & 0.9082 & 0.8804 \\

0.30 & 0.9754  & 0.9746  & 0.9638  & 0.9822 & 0.9814 & 0.9740 \\
\hline
 \end{tabular*}
\end{center}
\end{table}

Tables 1-6 illustrate that the three tests can control the nominal
size $\alpha=0.05$. Especially, when $p$ is larger than $n$,
empirical sizes are closer to the nominal level. Furthermore, the
test $\widehat{T}_{1}$ and $\widehat{T}_{2}$ have similar powers and
they are more powerful than $\widehat{T}_{H}$ for Cases 1-3.
According to our theoretical comparisons in the
Section~\ref{sec:4.1}, we observed that $\widehat{T}$ tends to have
more powers than $\widehat{T}_{H}$ when we assume homogeneous
covariance matrices. Similarly, our numerical result shows that
$\widehat{T}$ has more powers than $\widehat{T}_{H}$ even under
inhomogeneous covariance matrices when sample sizes are unbalanced.

We also collect the information on the ratios
$\widehat{{\rm{tr}}(\bm{\Sigma}_{l}^{2})}/{\rm{tr}}(\bm{\Sigma}_{l}^{2})$
and
$\widetilde{{\rm{tr}}(\bm{\Sigma}_{l}^{2})}/{\rm{tr}}(\bm{\Sigma}_{l}^{2})$,
respectively. Without loss of generality, we here select $l=1$ for
different cases, where $p=50, 200, 500$ and $1000$, and $n_{1}$ is
from $10$ to $160$ with adding $30$ each time, respectively. Table 7
reports the results of empirical averages and standard deviations of
ratios, respectively. It shows that the proposed estimator of
${\rm{tr}}(\bm{\Sigma}_{1}^{2})$ owns much smaller bias than that in
\cite{Hu:Bai:2015} in all cases. Meanwhile, standard deviations of
the new estimator are better than those of Hu's estimator in most
cases.

\begin{table}[H]
\tabcolsep 4mm\par{\footnotesize\textbf{Table 7} Empirical averages
of
$\widehat{{\rm{tr}}(\bm{\Sigma}_{1}^{2})}/{\rm{tr}}(\bm{\Sigma}_{1}^{2})$
(NEW) and
$\widetilde{{\rm{tr}}(\bm{\Sigma}_{1}^{2})}/{\rm{tr}}(\bm{\Sigma}_{1}^{2})$
in Hu et al. (2015) with standard deviations in the parentheses, respectively.}
\begin{center}
\def\temptablewidth{1\textwidth}
\begin{tabular*}{\temptablewidth}{@{\extracolsep{\fill}}lllll}
\hline
$p$ &  $n_{1}$   & {New} & HB & ${\rm{tr}}(\bm{\Sigma}_{1}^{2})$\\
\hline
50  &10  & 1.0807 (0.4031)  & 1.2045 (0.4650)  & 25000 \\

    &40 & 1.0850 (0.1575)  & 1.1211 (0.1682)  \\

    &70 & 1.0852 (0.1147)  & 1.1059 (0.1194)  \\

    &100 & 1.0849 (0.0958)  & 1.0998 (0.0986)  \\

    &130 & 1.0858 (0.0853)  & 1.0974 (0.0872)  \\

    &160 & 1.0828 (0.0738)  & 1.0921 (0.0751)  \\

200 &10  & 1.0836 (0.3212)  & 1.2058 (0.3316)  & 100900 \\

    &40 & 1.0854 (0.0955)  & 1.1220 (0.0986)  \\

    &70 & 1.0847 (0.0640)  & 1.1056 (0.0653)  \\

    &100 & 1.0839 (0.0519)  & 1.0987 (0.0530)  \\

    &130 & 1.0852 (0.0440)  & 1.0967 (0.0445)  \\

    &160 & 1.0837 (0.0390)  & 1.0931 (0.0394)  \\

500 &10  & 1.0773 (0.2930)  & 1.2029 (0.2947)  & 252700 \\

    &40 & 1.0861 (0.0747)  & 1.1220 (0.0756)  \\

    &70 & 1.0832 (0.0482)  & 1.1040 (0.0482)  \\

    &100 & 1.0837 (0.0366)  & 1.0987 (0.0367)  \\

    &130 & 1.0841 (0.0311)  & 1.0955 (0.0313)  \\

    &160 & 1.0840 (0.0272)  & 1.0932 (0.0273)  \\

1000 &10  & 1.0846 (0.2868)  & 1.2045 (0.2822)  & 505700 \\

    &40 & 1.0843 (0.0675)  & 1.1209 (0.0673)  \\

    &70 & 1.0845 (0.0407)  & 1.1054 (0.0405)  \\

    &100 & 1.0843 (0.0312)  & 1.0991 (0.0309)  \\

    &130 & 1.0844 (0.0258)  & 1.0957 (0.0256)  \\

    &160 & 1.0840 (0.0217)  & 1.0933 (0.0215)  \\
\hline
 \end{tabular*}
\end{center}
\end{table}

\section{Concluding remarks}\label{sec:5}
%In recently statistical research, testing hypothesis has become
%increasingly popular in high dimensional data.
In this paper, we
propose a new test for $k$ sample BF problem and provide the
theoretical results and numerical studies.
The new test procedure is
modified from $T_{S}$ in \cite{Schott:2007} under weaker conditions than those in \cite{Schott:2007}
and the proposed test has the same
asymptotic properties.
The theoretical results illustrate that our proposed test
has the same asymptotic power as that of $T_{H}$ for the case of balanced sample sizes.
The theoretical and numerical studies in this paper further show that our proposed
test can control the nominal level and  has larger powers than
those of $T_{H}$ in many cases of unbalanced sample sizes.
It is expected that the proposed test can detect the $H_1$ more efficiently than $T_H$ when
sample sizes are unbalanced.
%To summarize, the proposed
%test improves the existing test $T_{H}$ in power.

\section*{Acknowledgments}
Cao's research is supported by the National Natural Science
Foundation of China (No.~11526070, 11601008), Anhui Provincial
Natural Science Foundation (No.\ 1508085QA11) and Doctor Startup
Foundation of Anhui Normal University (No.~2016XJJ101), which
facilitated the research visit of the first author to University of
Maryland Baltimore County, USA. He's research is supported by the
National Natural Science Foundation of China (No.~11201005).

\appendix
\section*{Appendix}\label{Appendix}

\section{\noindent{\bf{Proof of Lemma \CG{3.1}}}}
\indent We only need to prove
${\rm{tr}}(\bm{S}_{ln_{l1}}\bm{S}_{ln_{l2}})/{\rm{tr}}(\bm{\Sigma}_{l}^{2})\overset{pr}\longrightarrow
1$ because the proof of ratio-consistent estimator of
${\rm{tr}}(\bm{\Sigma}_{l}\bm{\Sigma}_{s})$ is similar.

Note that
\begin{align*}
{\rm{tr}}(\bm{S}_{ln_{l1}}\bm{S}_{ln_{l2}})=\frac{1}{(n_{l1}-1)(n_{l2}-1)}\sum\limits_{i=1}^{n_{l1}}(\bm{Z}_{li}-\overline{\bm{Z}}_{ln_{l1}})^{\trans}\bm{\Gamma}_{l}^{\trans}\bm{U}\bm{\Gamma}_{l}(\bm{Z}_{li}-\overline{\bm{Z}}_{ln_{l1}})
\end{align*}
where
$\bm{U}:=\sum\limits_{j=n_{l1}+1}^{n_{l}}(\bm{X}_{lj}-\overline{\bm{X}}_{ln_{l2}})(\bm{X}_{lj}-\overline{\bm{X}}_{ln_{l2}})^{\trans}$
and $\overline{\bm{Z}}_{ln_{l1}}$ is the sample mean of
$\bm{Z}_{l1},\ldots,\bm{Z}_{ln_{l1}}$.

Thus, we have
\begin{align}\label{eq5}
{\rm{E}}\Big\{{\rm{tr}}(\bm{S}_{ln_{l1}}\bm{S}_{ln_{l2}})\Big\}={\rm{tr}}(\bm{\Sigma}_{l}^{2}).
\end{align}

In order to prove the conclusion of Lemma \CG{3.1}, it is sufficient
to prove
\begin{align}\label{eq6}
{\rm{Var}}\Big\{{\rm{tr}}(\bm{S}_{ln_{l1}}\bm{S}_{ln_{l2}})\Big\}=o\Big\{{\rm{tr}}^{2}(\bm{\Sigma}_{l}^{2})\Big\}.
\end{align}

It is easy to get
\begin{align*}
{\rm{Var}}\Big\{{\rm{tr}}(\bm{S}_{ln_{l1}}\bm{S}_{ln_{l2}})\Big\}&={\rm{Var}}\Big\{{\rm{E}}[{\rm{tr}}(\bm{S}_{ln_{l1}}\bm{S}_{ln_{l2}})|\bm{U}]\Big\}+{\rm{E}}\Big\{{\rm{Var}}[{\rm{tr}}(\bm{S}_{ln_{l1}}\bm{S}_{ln_{l2}})|\bm{U}]\Big\}\\
&=:\frac{1}{(n_{l1}-1)^{2}(n_{l2}-1)^{2}}(\mbox{I}+\mbox{II}).
\end{align*}

Furthermore, we have
\begin{align}\label{eq7}
\mbox{I}&=(n_{l1}-1)^{2}{\rm{Var}}\left\{\sum\limits_{j=n_{l1}+1}^{n_{l}}(\bm{Z}_{lj}-\overline{\bm{Z}}_{ln_{l2}})^{\trans}\bm{\Gamma}_{l}^{\trans}\bm{\Sigma}_{l}\bm{\Gamma}_{l}(\bm{Z}_{lj}-\overline{\bm{Z}}_{ln_{l2}})\right\}\nonumber\\
&=(n_{l1}-1)^{2}\left\{\frac{\gamma(n_{l2}-1)^{2}}{n_{l2}}{\rm{tr}}[\bm{\Gamma}_{l}^{\trans}\bm{\Sigma}_{l}{\bm{\Gamma}_{l}}{\rm{diag}}(\bm{\Gamma}_{l}^{\trans}\bm{\Sigma}_{l}{\bm{\Gamma}_{l}})]+2(n_{l2}-1){\rm{tr}}(\bm{\Sigma}_{l}^{4})\right\}\nonumber\\
&{\leq}(n_{l1}-1)^{2}\left\{\frac{\gamma(n_{l2}-1)^{2}}{n_{l2}}{\rm{tr}}^{2}(\bm{\Sigma}_{l}^{2})+2(n_{l2}-1){\rm{tr}}(\bm{\Sigma}_{l}^{4})\right\}\nonumber\\
&=O(n^{3}){\rm{tr}}^{2}\left(\sum\limits_{i=1}^{k}\bm{\Sigma}_{i}\right)^{2}.
\end{align}

\begin{align}\label{eq8}
\mbox{II}&={\rm{E}}\left\{\frac{\gamma(n_{l1}-1)^{2}}{n_{l1}}{\rm{tr}}[\bm{\Gamma}_{l}^{\trans}\bm{U}{\bm{\Gamma}_{l}}{\rm{diag}}(\bm{\Gamma}_{l}^{\trans}\bm{U}{\bm{\Gamma}_{l}})]+2(n_{l1}-1){\rm{tr}}(\bm{\Gamma}_{l}^{\trans}\bm{U}{\bm{\Gamma}_{l}})^{2}\right\}\nonumber\\
&{\leq}\frac{(n_{l1}-1)(n_{l2}-1)[\gamma(n_{l2}-1)(2n_{l1}+{\gamma}n_{l1}-\gamma+n_{l1}n_{l2}-n_{l2})+2n_{l1}n_{l2}]}{n_{l1}n_{l2}}{\rm{tr}}^{2}(\bm{\Sigma}_{l}^{2})\nonumber\\
&\quad+\frac{2(n_{l1}-1)(n_{l2}-1)(n_{l1}n_{l2}+\gamma(n_{l1}-1))}{n_{l1}}{\rm{tr}}(\bm{\Sigma}_{l}^{4})\nonumber\\
&=O(n^{3}){\rm{tr}}^{2}\left(\sum\limits_{i=1}^{k}\bm{\Sigma}_{i}\right)^{2}.
\end{align}

Equations (\ref{eq7}) and (\ref{eq8}) show that equation (\ref{eq6})
holds. Thus, the proof of lemma is completed.

For convenience, let
$\bm{C}_{j+\begin{smallmatrix}\sum\limits_{i=1}^{l-1}\end{smallmatrix}n_{i}}=\bm{X}_{lj}$
for $j=1,\ldots,n_{l}$ with $\sum\limits_{i=1}^{0}n_{i}=0$, and
\begin{numcases}
{\eta_{ij}=}
\frac{n-n_{l}}{n(n_{l}-1)}\bm{C}_{i}^{\prime}\bm{C}_{j}, &$i,j\in\Lambda_{l}$, $l=1,\ldots,k$, \nonumber\\
-\frac{1}{n}\bm{C}_{i}^{\prime}\bm{C}_{j},
&$(i,j)\in\Lambda_{l}\times \Lambda_{s}$, $1\leq
l<s\leq{k}$,\nonumber
\end{numcases}
where
$\Lambda_{l}=\left\{\sum\limits_{j=1}^{l-1}n_{j}+1,\sum\limits_{j=1}^{l-1}n_{j}+2,\cdots,\sum\limits_{j=1}^{l}n_{j}\right\}$
with $l=1,\ldots,k$.

Define further $D_{j}=\sum\limits_{i=1}^{j-1}\eta_{ij}$,
$F_{m}=\sum\limits_{j=2}^{m}D_{j}$ and
$\mathscr{C}_{m}=\sigma(\bm{C}_{1},\ldots,\bm{C}_{m})$ which is the
$\sigma$-field generated by $\bm{C}_{1},\ldots,\bm{C}_{m}$.

Combining $\bm{C}_{j}$, $\eta_{ij}$ with $D_{j}$, we have
\begin{align*}
T&=2\sum_{l=1}^{k}\sum_{j=2+\sum\limits_{m=1}^{l-1}n_{m}}^{\sum\limits_{m=1}^{l}n_{m}}\sum_{i=1+\sum\limits_{m=1}^{l-1}n_{m}}^{j-1}\eta_{ij}
+2\sum_{l=1}^{k-1}\sum_{s=l+1}^{k}\sum_{i=1+\sum\limits_{m=1}^{l-1}n_{m}}^{\sum\limits_{m=1}^{l}n_{m}}\sum_{j=1+\sum\limits_{m=1}^{s-1}n_{m}}^{\sum\limits_{m=1}^{s}n_{m}}\eta_{ij}\\
&=2\sum\limits_{j=2}^{n}\sum\limits_{i=1}^{j-1}\eta_{ij}=\sum\limits_{j=2}^{n}D_{j}.
\end{align*}

In order to prove our main results, the following lemmas are firstly
given. Without loss of generality, we here assume that
$\bm{\mu}_{1}=\cdots=\bm{\mu}_{k}=0$ in the process of proving
lemmas.

\begin{lemma}\label{lmA.1}
$\{D_{j},\mathscr{C}_{j}\}_{j=1}^{n}$ is the sequence of zero mean
and a square integrable martingale for all $n$.
\end{lemma}

\begin{proof}
Firstly, we have $\mathscr{C}_{1}\subseteq \cdots \subseteq
\mathscr{C}_{n}$, and $\{D_{j},\mathscr{C}_{j}\}_{j=1}^{n}$ is a
square integrable sequence with zero mean. Hence, we only need to
prove ${\rm{E}}(F_{m}|\mathscr{C}_{j})=F_{j}$ for $\forall m\geq j$.

Notice that ${\rm{E}}(D_{q}|\mathscr{C}_{j})=0$ for $\forall q > j$.
Therefore,
\begin{align*}
{\rm{E}}(F_{m}|\mathscr{C}_{j})&=F_{j}+{\rm{E}}(\sum\limits_{q=j+1}^{m}D_{q}|\mathscr{C}_{j})\\
&=F_{j}+\sum\limits_{q=j+1}^{m}{\rm{E}}(D_{q}|\mathscr{C}_{j})=F_{j}.
\end{align*}
Lastly, this completes the proof of Lemma \CG{A.1}.
\end{proof}

\begin{lemma}\label{lmA.2}
Under the assumptions of (\CG{A1}) and (\CG{A3}), as $n,
p\to\infty$, it gets
\begin{align*}
\sum\limits_{j=2}^{n}{\rm{E}}(D_{j}^{2}|\mathscr{C}_{j-1})\overset{pr}\longrightarrow
\frac{1}{4}\sigma_{T}^{2}.
\end{align*}
\end{lemma}

\begin{proof}
For $\forall$$j\in\Lambda_{l}$,
\begin{align}\label{eq9}
{\rm{E}}(D_{j}^{2}|\mathscr{C}_{j-1})&={\rm{E}}\left\{\left(\sum\limits_{i=1}^{j-1}\eta_{ij}\right)^{2}\Big|\mathscr{C}_{j-1}\right\}\nonumber\\
&=n^{-2}{\rm{E}}\left\{\bm{H}_{j-1}^{\trans}\bm{C}_{j}\bm{C}_{j}^{\trans}\bm{H}_{j-1}\big|\mathscr{C}_{j-1}\right\}=n^{-2}\bm{H}_{j-1}^{\trans}{\bm{\Sigma}_{l}}\bm{H}_{j-1},
\end{align}
where
$\bm{H}_{j-1}=\sum\limits_{i=1+\begin{smallmatrix}\sum\limits_{s=1}^{l-1}\end{smallmatrix}n_{s}}^{j-1}\frac{n-n_{l}}{n_{l}-1}\bm{C}_{i}-\sum\limits_{i=1}^{\begin{smallmatrix}\sum\limits_{s=1}^{l-1}\end{smallmatrix}n_{s}}\bm{C}_{i}$.

If we eefine
$G_{l}=\sum\limits_{j\in\Lambda_{l}}{\rm{E}}(D_{j}^{2}|\mathscr{C}_{j-1})$,
then we get from (\ref{eq9})
\begin{align}\label{eq10}
{\rm{E}}(G_{l})&=n^{-2}\sum\limits_{j\in\Lambda_{l}}\left\{\frac{(n-n_{l})^{2}}{(n_{l}-1)^{2}}\left(j-1-\sum\limits_{s=1}^{l-1}n_{s}\right){\rm{tr}}(\bm{\Sigma}_{l}^{2})+\sum\limits_{s=1}^{l-1}n_{s}{\rm{tr}}(\bm{\Sigma}_{l}\bm{\Sigma}_{s})\right\}\nonumber\\
&=n^{-2}\left\{\frac{n_{l}(n-n_{l})^{2}}{2(n_{l}-1)}{\rm{tr}}(\bm{\Sigma}_{l}^{2})+n_{l}\sum\limits_{s=1}^{l-1}n_{s}{\rm{tr}}(\bm{\Sigma}_{l}\bm{\Sigma}_{s})\right\}.
\end{align}
Therefore,
\begin{align*}
{\rm{E}}\left\{\sum\limits_{j=2}^{n}{\rm{E}}(D_{j}^{2}|\mathscr{C}_{j-1})\right\}=\sum\limits_{l=1}^{k}{\rm{E}}(G_{l})=\frac{1}{4}\sigma_{T}^{2}.
\end{align*}

On the other hand, we have
\begin{align}\label{eq11}
{\rm{E}}(G_{l}^{2})&={\rm{E}}\left\{\sum\limits_{j\in\Lambda_{l}}{\rm{E}}(D_{j}^{2}|\mathscr{C}_{j-1})\right\}^{2}=n^{-4}{\rm{E}}\left\{\sum\limits_{j\in\Lambda_{l}}\bm{H}_{j-1}^{\trans}{\bm{\Sigma}_{l}}\bm{H}_{j-1}\right\}^{2}\nonumber\\
&=n^{-4}{\rm{E}}\left\{\sum\limits_{j\in\Lambda_{l}}\left(\bm{H}_{j-1}^{\trans}{\bm{\Sigma}_{l}}\bm{H}_{j-1}\right)^{2}+\sum\limits_{j\neq{h}\in\Lambda_{l}}\bm{H}_{j-1}^{\trans}{\bm{\Sigma}_{l}}\bm{H}_{j-1}\bm{H}_{h-1}^{\trans}{\bm{\Sigma}_{l}}\bm{H}_{h-1}\right\}\nonumber\\
&=:n^{-4}(\mbox{III}+\mbox{IV}).
\end{align}

Further calculations result in
\begin{align}\label{eq12}
\mbox{III}&=\frac{{\gamma}n_{l}(n-n_{l})^{4}}{2(n_{l}-1)^{3}}{\rm{tr}}\left\{\bm{\Gamma}_{l}^{\trans}\bm{\Sigma}_{l}\bm{\Gamma}_{l}{\rm{diag}}(\bm{\Gamma}_{l}^{\trans}\bm{\Sigma}_{l}\bm{\Gamma}_{l})\right\}+\frac{n_{l}(2n_{l}-1)(n-n_{l})^{4}}{6(n_{l}-1)^{3}}\left\{2{\rm{tr}}(\bm{\Sigma}_{l}^{4})+{\rm{tr}}^{2}(\bm{\Sigma}_{l}^{2})\right\}\nonumber\\
&\quad+{\gamma}n_{l}\sum\limits_{s=1}^{l-1}n_{s}{\rm{tr}}\left\{\bm{\Gamma}_{s}^{\trans}\bm{\Sigma}_{l}\bm{\Gamma}_{s}{\rm{diag}}(\bm{\Gamma}_{s}^{\trans}\bm{\Sigma}_{l}\bm{\Gamma}_{s})\right\}+n_{l}\sum\limits_{s=1}^{l-1}n_{s}^{2}\left\{2{\rm{tr}}(\bm{\Sigma}_{l}\bm{\Sigma}_{s})^{2}+{\rm{tr}}^{2}(\bm{\Sigma}_{l}\bm{\Sigma}_{s})\right\}\nonumber\\
&\quad+n_{l}\sum\limits_{s{\neq}h}^{l-1}n_{s}n_{h}{\rm{tr}}(\bm{\Sigma}_{l}\bm{\Sigma}_{s}){\rm{tr}}(\bm{\Sigma}_{l}\bm{\Sigma}_{h})+2n_{l}\sum\limits_{s{\neq}h}^{l-1}n_{s}n_{h}{\rm{tr}}(\bm{\Sigma}_{l}\bm{\Sigma}_{s}\bm{\Sigma}_{l}\rm{\Sigma}_{h})\nonumber\\
&\quad+\frac{n_{l}(n-n_{l})^{2}}{n_{l}-1}\sum\limits_{s=1}^{l-1}n_{s}{\rm{tr}}(\bm{\Sigma}_{l}\bm{\Sigma}_{s}){\rm{tr}}(\bm{\Sigma}_{l}^{2})+\frac{2n_{l}(n-n_{l})^{2}}{n_{l}-1}\sum\limits_{s=1}^{l-1}n_{s}{\rm{tr}}(\bm{\Sigma}_{l}^{3}\bm{\Sigma}_{s}),
\end{align}
\begin{align}\label{eq13}
\mbox{IV}&=2\sum\limits_{j<h\in\Lambda_{l}}{\rm{E}}\left(\bm{H}_{j-1}^{\trans}{\bm{\Sigma}_{l}}\bm{H}_{j-1}\bm{H}_{h-1}^{\trans}{\bm{\Sigma}_{l}}\bm{H}_{h-1}\right)\nonumber\\
&=2\sum\limits_{j<h\in\Lambda_{l}}{\rm{E}}\left(\bm{H}_{j-1}^{\trans}{\bm{\Sigma}_{l}}\bm{H}_{j-1}\right)^{2}+4\sum\limits_{j<h\in\Lambda_{l}}{\rm{E}}\left\{\bm{H}_{j-1}^{\trans}{\bm{\Sigma}_{l}}\bm{H}_{j-1}\bm{H}_{j-1}^{\trans}\bm{\Sigma}_{l}\left(\sum\limits_{i=j}^{h-1}\frac{n-n_{l}}{n_{l}-1}\bm{C}_{i}\right)\right\}\nonumber\\
&\quad+2\sum\limits_{j<h\in\Lambda_{l}}{\rm{E}}\left\{\bm{H}_{j-1}^{\trans}{\bm{\Sigma}_{l}}\bm{H}_{j-1}\left(\sum\limits_{i=j}^{h-1}\frac{n-n_{l}}{n_{l}-1}\bm{C}_{i}\right)^{\trans}{\bm{\Sigma}_{l}}\left(\sum\limits_{i=j}^{h-1}\frac{n-n_{l}}{n_{l}-1}\bm{C}_{i}\right)\right\}\nonumber\\
&=:\mbox{V}+\mbox{VI}+\mbox{VII},
\end{align}
where
\begin{align}\label{eq14}
\mbox{V}&=\frac{{\gamma}n_{l}(n_{l}-2)(n-n_{l})^{4}}{3(n_{l}-1)^{3}}{\rm{tr}}\left\{\bm{\Gamma}_{l}^{\trans}\bm{\Sigma}_{l}\bm{\Gamma}_{l}{\rm{diag}}(\bm{\Gamma}_{l}^{\trans}\bm{\Sigma}_{l}\bm{\Gamma}_{l})\right\}+\frac{n_{l}(n_{l}-2)(n-n_{l})^{4}}{6(n_{l}-1)^{2}}\big\{2{\rm{tr}}(\bm{\Sigma}_{l}^{4})\nonumber\\
&\quad+{\rm{tr}}^{2}(\bm{\Sigma}_{l}^{2})\big\}+{\gamma}n_{l}(n_{l}-1)\sum\limits_{s=1}^{l-1}n_{s}{\rm{tr}}\left\{\bm{\Gamma}_{s}^{\trans}\bm{\Sigma}_{l}\bm{\Gamma}_{s}{\rm{diag}}(\bm{\Gamma}_{s}^{\trans}\bm{\Sigma}_{l}\bm{\Gamma}_{s})\right\}\nonumber\\
&\quad+n_{l}(n_{l}-1)\sum\limits_{s{\neq}h}^{l-1}n_{s}n_{h}{\rm{tr}}(\bm{\Sigma}_{l}\bm{\Sigma}_{s}){\rm{tr}}(\bm{\Sigma}_{l}\bm{\Sigma}_{h})+2n_{l}(n_{l}-1)\sum\limits_{s{\neq}h}^{l-1}n_{s}n_{h}{\rm{tr}}(\bm{\Sigma}_{l}\bm{\Sigma}_{s}\bm{\Sigma}_{l}\bm{\Sigma}_{h})\nonumber\\
&\quad+\frac{2n_{l}(n_{l}-2)(n-n_{l})^{2}}{3(n_{l}-1)}\sum\limits_{s=1}^{l-1}n_{s}{\rm{tr}}(\bm{\Sigma}_{l}\bm{\Sigma}_{s}){\rm{tr}}(\bm{\Sigma}_{l}^{2})+\frac{4n_{l}(n_{l}-2)(n-n_{l})^{2}}{3(n_{l}-1)}\sum\limits_{s=1}^{l-1}n_{s}{\rm{tr}}(\bm{\Sigma}_{l}^{3}\bm{\Sigma}_{s})\nonumber\\
&\quad+n_{l}(n_{l}-1)\sum\limits_{s=1}^{l-1}n_{s}^{2}\left\{2{\rm{tr}}(\bm{\Sigma}_{l}\bm{\Sigma}_{s})^{2}+{\rm{tr}}^{2}(\bm{\Sigma}_{l}\bm{\Sigma}_{s})\right\}
\end{align}
and
\begin{align}\label{eq15}
\mbox{VII}=\frac{n_{l}(n_{l}-2)(n_{l}+1)(n-n_{l})^{4}}{12(n_{l}-1)^{3}}{\rm{tr}}^{2}(\bm{\Sigma}_{l}^{2})+\frac{n_{l}(n_{l}+1)(n-n_{l})^{2}}{3(n_{l}-1)}\sum\limits_{s=1}^{l-1}n_{s}{\rm{tr}}(\bm{\Sigma}_{l}\bm{\Sigma}_{s}){\rm{tr}}(\bm{\Sigma}_{l}^{2}).
\end{align}

Thus by equations (\ref{eq10})-(\ref{eq15}) and $\mbox{VI}=0$, we
obtain
\begin{align*}
{\rm{Var}}(G_{l})&=\frac{{\gamma}n_{l}(2n_{l}-1)(n-n_{l})^{4}}{6n^{4}(n_{l}-1)^{3}}{\rm{tr}}\left\{\bm{\Gamma}_{l}^{\trans}\bm{\Sigma}_{l}\bm{\Gamma}_{l}{\rm{diag}}(\bm{\Gamma}_{l}^{\trans}\bm{\Sigma}_{l}\bm{\Gamma}_{l})\right\}\\
&\quad+\frac{{\gamma}n_{l}^{2}}{n^{4}}\sum\limits_{s=1}^{l-1}n_{s}{\rm{tr}}\left\{\bm{\Gamma}_{s}^{\trans}\bm{\Sigma}_{l}\bm{\Gamma}_{s}{\rm{diag}}(\bm{\Gamma}_{s}^{\trans}\bm{\Sigma}_{l}\bm{\Gamma}_{s})\right\}\\
&\quad+\frac{n_{l}(n_{l}^{2}-n_{l}+1)(n-n_{l})^{4}}{3n^{4}(n_{l}-1)^{3}}{\rm{tr}}(\bm{\Sigma}_{l}^{4})+\frac{2n_{l}^{2}}{n^{4}}\sum\limits_{s{\neq}h}^{l-1}n_{s}n_{h}{\rm{tr}}(\bm{\Sigma}_{l}\bm{\Sigma}_{s}\bm{\Sigma}_{l}\bm{\Sigma}_{h})\\
&\quad+\frac{2n_{l}(2n_{l}-1)(n-n_{l})^{2}}{3n^{4}(n_{l}-1)}\sum\limits_{s=1}^{l-1}n_{s}{\rm{tr}}(\bm{\Sigma}_{l}^{3}\bm{\Sigma}_{s})+\frac{2n_{l}^{2}}{n^{4}}\sum\limits_{s=1}^{l-1}n_{s}^{2}{\rm{tr}}(\bm{\Sigma}_{l}\bm{\Sigma}_{s})^{2}.
\end{align*}

Note that
\begin{align}\label{eq16}
{\rm{tr}}\left\{\bm{\Gamma}_{s}^{\trans}\bm{\Sigma}_{l}\bm{\Gamma}_{s}{\rm{diag}}(\bm{\Gamma}_{s}^{\trans}\bm{\Sigma}_{l}\bm{\Gamma}_{s})\right\}{\leq}{\rm{tr}}^{2}(\bm{\Sigma}_{l}\bm{\Sigma}_{s})\leq{{\rm{tr}}^{2}\left(\sum\limits_{i=1}^{k}\bm{\Sigma}_{i}\right)^{2}}.
\end{align}

So, it easily gets, from (\CG{A1}), (\CG{A3}) and equation
(\ref{eq16}),
\begin{align*}
{\rm{Var}}(G_{l})=o(\sigma_{T}^{4}).
\end{align*}

Lastly, via Cauchy-Schwarz inequality
${\rm{Cov}}^{2}(G_{l},G_{s}){\leq}{\rm{Var}}(G_{l}){\rm{Var}}(G_{s})$,
we have
\begin{align*}
{\rm{Var}}\left\{\sum\limits_{j=2}^{n}{\rm{E}}(D_{j}^{2}|\mathscr{C}_{j-1})\right\}={\rm{Var}}\left(\sum\limits_{l=1}^{k}G_{l}\right)=o(\sigma_{T}^{4}).
\end{align*}
This completes the proof of lemma \CG{A.2}.
\end{proof}

Now attentions are paid to proving the Lindeberg condition.

\begin{lemma}\label{lm4.3}
Under the assumptions of (\CG{A1}) and (\CG{A3}), for
$\forall\epsilon > 0$, as $n, p\to\infty$, it takes
\begin{align*}
\sum\limits_{j=2}^{n}\sigma_{T}^{-2}{\rm{E}}\left\{D_{j}^{2}I(|D_{j}|>\epsilon\sigma_{T})|\mathscr{C}_{j-1}\right\}\overset{pr}\longrightarrow
0.
\end{align*}
\end{lemma}

\begin{proof}
Firstly, it is easy to obtain
\begin{align*}
\sum\limits_{j=2}^{n}\sigma_{T}^{-2}{\rm{E}}\left\{D_{j}^{2}I(|D_{j}|>\epsilon\sigma_{T})|\mathscr{C}_{j-1}\right\}\leq\sum\limits_{j=2}^{n}\epsilon^{-2}\sigma_{T}^{-4}{\rm{E}}(D_{j}^{4}|\mathscr{C}_{j-1}).
\end{align*}
So, we only need to prove
\begin{align}\label{eq17}
{\rm{E}}\left\{\sum\limits_{j=2}^{n}{\rm{E}}(D_{j}^{4}|\mathscr{C}_{j-1})\right\}=o(\sigma_{T}^{4}).
\end{align}

Let
$K_{l}=\sum\limits_{j\in\Lambda_{l}}{\rm{E}}(D_{j}^{4}|\mathscr{C}_{j-1})$,
then
$\sum\limits_{j=2}^{n}{\rm{E}}(D_{j}^{4}|\mathscr{C}_{j-1})=\sum\limits_{l=1}^{k}K_{l}$.
Further calculations can lead to
\begin{align*}
K_{l}&=n^{-4}\sum\limits_{j\in\Lambda_{l}}{\rm{E}}\left\{(\bm{C}_{j}^{\trans}\bm{H}_{j-1}\bm{H}_{j-1}^{\trans}\bm{C}_{j})^{2}|\mathscr{C}_{j-1}\right\}\\
&=n^{-4}\sum\limits_{j\in\Lambda_{l}}\Big\{\gamma{\rm{tr}}\left[\bm{\Gamma}_{l}^{\trans}\bm{H}_{j-1}\bm{H}_{j-1}^{\trans}{\bm{\Gamma}_{l}}{\rm{diag}}(\bm{\Gamma}_{l}^{\trans}\bm{H}_{j-1}\bm{H}_{j-1}^{\trans}{\bm{\Gamma}_{l}})\right]\\
&\quad+2{\rm{tr}}(\bm{\Gamma}_{l}^{\trans}\bm{H}_{j-1}\bm{H}_{j-1}^{\trans}{\bm{\Gamma}_{l}})^{2}+(\bm{H}_{j-1}^{\trans}{\bm{\Sigma}_{l}}\bm{H}_{j-1})^{2}\Big\}\\
&{\leq}(3+\gamma)n^{-4}\sum\limits_{j\in\Lambda_{l}}(\bm{H}_{j-1}^{\trans}{\bm{\Sigma}_{l}}\bm{H}_{j-1})^{2},
\end{align*}
where the last inequality is based on
${\rm{tr}}(\bm{\Gamma}_{l}^{\trans}\bm{H}_{j-1}\bm{H}_{j-1}^{\trans}{\bm{\Gamma}_{l}})^{2}{\leq}(\bm{H}_{j-1}^{\trans}{\bm{\Sigma}_{l}}\bm{H}_{j-1})^{2}$
and
\begin{align*}
{\rm{tr}}\left\{\bm{\Gamma}_{l}^{\trans}\bm{H}_{j-1}\bm{H}_{j-1}^{\trans}{\bm{\Gamma}_{l}}{\rm{diag}}(\bm{\Gamma}_{l}^{\trans}\bm{H}_{j-1}\bm{H}_{j-1}^{\trans}{\bm{\Gamma}_{l}})\right\}{\leq}{\rm{tr}}^{2}(\bm{\Gamma}_{l}^{\trans}\bm{H}_{j-1}\bm{H}_{j-1}^{\trans}{\bm{\Gamma}_{l}})=(\bm{H}_{j-1}^{\trans}{\bm{\Sigma}_{l}}\bm{H}_{j-1})^{2}.
\end{align*}

According to equations (\ref{eq11}) and (\ref{eq12}), we get
\begin{align*}
{\rm{E}}(K_{l})&{\leq}(3+\gamma)n^{-4}\mbox{III}=o(\sigma_{T}^{4}),
\end{align*}
which implies
\begin{align*}
{\rm{E}}\left\{\sum\limits_{j=2}^{n}{\rm{E}}(D_{j}^{4}|\mathscr{C}_{j-1})\right\}=\sum\limits_{l=1}^{k}{\rm{E}}(K_{l})=o(\sigma_{T}^{4}).
\end{align*}
Then the required result follows.
\end{proof}

\section{{\bf{Proof of Theorem \CG{3.1}}}}
\indent From Lemmas \CG{A.1}-\CG{A.3}, as $n, p\to\infty$, we have
\begin{align*}
\frac{T}{\sigma_{T}}\overset{d}\longrightarrow N(0,1).
\end{align*}

By Lemmas \CG{3.1} and \CG{3.2}, we obtain
\begin{align*}
\widehat{T}_{1}=\frac{T}{\sigma_{T}}\frac{\sigma_{T}}{\widehat{\sigma_{T}}}\overset{d}\longrightarrow
N(0,1)~~~\mbox{and}~~~
\widehat{T}_{2}=\frac{T}{\sigma_{T}}\frac{\sigma_{T}}{\widetilde{\sigma_{T}}}\overset{d}\longrightarrow
N(0,1)
\end{align*}
from Slutsky's Theorem (Ferguson (1996)).

\section{{\bf{Proof of Theorem \CG{3.2}}}}
\indent According to Lemmas \CG{A.1}-\CG{A.3}, as $n, p\to\infty$,
we get
\begin{align*}
\frac{T-{\rm{E}}(T)}{\sigma_{T}}\overset{d}\longrightarrow N(0,1).
\end{align*}

On the other hand, we have
\begin{align*}
\frac{{\rm{E}}(T)}{\sigma_{T}}-\frac{\dfrac{n}{\sqrt{2}}\sum\limits_{l=1}^{k}\lambda_{l}(\mu_{l}-\widetilde{\mu})^{\trans}(\mu_{l}-\widetilde{\mu})}{\sqrt{\sum\limits_{l=1}^{k}(1-\lambda_{l})^{2}{\rm{tr}}(\bm{\Sigma}_{l}^{2})+\sum\limits_{l\neq{s}}^{k}\lambda_{l}\lambda_{s}{\rm{tr}}(\bm{\Sigma}_{l}\bm{\Sigma}_{s})}}\longrightarrow
0.
\end{align*}

Therefore, as $n, p\to\infty$, we obtain
\begin{align*}
P(\widehat{T}\geq\xi_{\alpha})-\Phi\left\{-\xi_{\alpha}+\frac{\dfrac{n}{\sqrt{2}}\sum\limits_{l=1}^{k}\lambda_{l}(\mu_{l}-\widetilde{\mu})^{\trans}(\mu_{l}-\widetilde{\mu})}{\sqrt{\sum\limits_{l=1}^{k}(1-\lambda_{l})^{2}{\rm{tr}}(\bm{\Sigma}_{l}^{2})+\sum\limits_{l\neq{s}}^{k}\lambda_{l}\lambda_{s}{\rm{tr}}(\bm{\Sigma}_{l}\bm{\Sigma}_{s})}}\right\}\longrightarrow
0,
\end{align*}
which implies the result is right.

\end{document}